\documentclass[12pt,a4paper,reqno]{amsart}

\usepackage{a4wide}
\usepackage{amssymb,amsmath,amsthm,mathrsfs}
\usepackage{graphicx}
\usepackage{float}
\usepackage{colortbl}
\usepackage{enumerate}
\usepackage{upref}
\usepackage[bookmarks=false]{hyperref}

\usepackage[T2A,T1]{fontenc}
\DeclareSymbolFont{cyrillic}{T2A}{cmr}{m}{n}
\DeclareMathSymbol{\D}{\mathalpha}{cyrillic}{196}

\theoremstyle{plain}
\newtheorem{theorem}{Theorem}[section]
\newtheorem{lemma}[theorem]{Lemma}
\newtheorem{corollaryP}[theorem]{Corollary}
\newtheorem{propositionP}[theorem]{Proposition}
\newtheorem{proposition}[theorem]{Proposition}
\newtheorem{corollary}[theorem]{Corollary}

\theoremstyle{definition}

\newtheorem*{condition}{Condition}

\theoremstyle{remark}
\newtheorem{remark}[theorem]{Remark}
\newtheorem*{note}{Note}

\usepackage{enumitem}
\makeatletter
\def\namedlabel#1#2{\begingroup
   #2%
 \def\@currentlabel{#2}%
   \phantomsection\label{#1}\endgroup
}

\def\R{\ensuremath{\mathbb R}}
\def\N{\ensuremath{\mathbb N}}

\def\I{\ensuremath{{\bf 1}}}
\def\e{{\ensuremath{\rm e}}}

\def\B{\ensuremath{\mathcal B}}

\def\l{\ensuremath{\left(}}
\def\r{\ensuremath{\right)}}
\def\P{\ensuremath{\mathcal P}}
\def\p{\ensuremath{\mathbb P}}

\def\o{\ensuremath{\text{o}}}

\def\A{\ensuremath{A^{(q)}}}
\def\n{\ensuremath{n}}

\def\X{\mathcal{X}}
\def\Y{\mathcal{Y}}

\def\ie{{\em i.e.}, }

\def\eps{\varepsilon}

\def\cv{\ensuremath{\text {Cor}}}
\newcommand{\dif}{\mathrm{d}}

\def\FS{\mathcal{FS}}

\numberwithin{equation}{section}

\begin{document}

\title{Speed of convergence for laws of rare events and escape rates}

\author[A. C. M. Freitas]{Ana Cristina Moreira Freitas}
\address{Ana Cristina Moreira Freitas\\ Centro de Matem\'{a}tica \&
Faculdade de Economia da Universidade do Porto\\ Rua Dr. Roberto Frias \\
4200-464 Porto\\ Portugal} \email{\href{mailto:amoreira@fep.up.pt}{amoreira@fep.up.pt}}
\urladdr{\url{http://www.fep.up.pt/docentes/amoreira/}}

\author[J. M. Freitas]{Jorge Milhazes Freitas}
\address{Jorge Milhazes Freitas\\ Centro de Matem\'{a}tica \& Faculdade de Ci\^encias da Universidade do Porto\\ Rua do
Campo Alegre 687\\ 4169-007 Porto\\ Portugal}
\email{\href{mailto:jmfreita@fc.up.pt}{jmfreita@fc.up.pt}}
\urladdr{\url{http://www.fc.up.pt/pessoas/jmfreita/}}

\author[M. Todd]{Mike Todd}
\address{Mike Todd\\ Mathematical Institute\\
University of St Andrews\\
North Haugh\\
St Andrews\\
KY16 9SS\\
Scotland} \email{\href{mailto:mjt20@st-andrews.ac.uk}{mjt20@st-andrews.ac.uk}}
\urladdr{\url{http://www.mcs.st-and.ac.uk/~miket/}}

\thanks{ACMF was partially supported by FCT grant SFRH/BPD/66174/2009. JMF was partially supported by FCT grant SFRH/BPD/66040/2009.  MT was partially supported by NSF grant DMS 1109587.  All authors are supported by FCT (Portugal) projects PTDC/MAT/099493/2008 and PTDC/MAT/120346/2010, which are financed by national and European structural funds through the programs  FEDER and COMPETE . All three authors were also supported by CMUP, which is financed by FCT (Portugal) through the programs POCTI and POSI, with national
and European structural funds. }

\date{\today}

\keywords{Extreme Value Theory, Return Time Statistics, Stationary Stochastic Processes, Metastability} \subjclass[2000]{37A50, 60G70, 37B20, 60G10, 37C25.}


\begin{abstract}
We obtain error terms on the rate of convergence to Extreme Value Laws, and to the asymptotic Hitting Time Statistics, for a general class of weakly dependent stochastic processes.
 The dependence of the error terms on the `time' and `length' scales is very explicit.  Specialising to data derived from a class of dynamical systems we find even more detailed error terms, one application of which is to consider escape rates through small holes in these systems.
\end{abstract}

\maketitle

\section{Introduction}
\label{sec:intro}

The study of the statistics of extreme events is both of classical importance, and a crucial topic across contemporary science.  Classically, the underlying stochastic processes are assumed to be independently distributed, but many more recent developments in this topic relate to the study of Extreme Value Laws (EVL) for dependent systems.  A standard approach to prove the existence of EVL in this setting is to check some conditions on the underlying process, for example Leadbetter's conditions $D(u_n)$ and $D'(u_n)$ in \cite{L73}.  Inspired by \cite{C01}, in a series of works  \cite{FF08, FFT10, FFT12, FFT13} the authors developed these conditions so that they had wider application, the main motivation being to stochastic processes coming from dynamical systems.  A natural question to now ask is: how fast is the convergence to the EVL?  (To rephrase, at a given finite stage in the process, what is the difference, or error term, between the law observed up to this time and the asymptotic law?)  For example, if the convergence were to be very slow, in a simulation the laws would be essentially invisible.    In this paper we address this question, particularly in the context of the latter papers above.

Error terms in the i.i.d.\ context are rather well-known, see for example \cite{HW79, S82} and the discussion in \cite[Section 2.4]{R08}.   However, the literature on dependent processes is much less extensive (for one case, see \cite{MS01}). On the other hand, if we think of our stochastic process as coming from a Markov chain or, more generally, a dynamical system, then there is an equivalence between EVL and Hitting Time Statistics (HTS) (see \cite{FFT10}), which then yields a significant body of literature coming from that side on these error terms \cite{GS97, A01, A04, AV09, AS11, K12}.  In this paper, taking inspiration from all these areas, we obtain sophisticated estimates on rates of convergence, where the dependence on time and `length' (i.e., the distance from the maximum) scales is made explicit.  We will first give general error terms under very general mixing conditions (in a way that unifies both clustering and non-clustering cases), and then impose some stronger conditions on our underlying process to obtain better estimates.

\subsection{A more technical introduction}

Let $X_0, X_1, \ldots$ be a stationary stochastic process, where each random variable (r.v.) $X_i:\Y\to\R$ is defined on the measure space $(\Y,\mathcal B,\p)$.
We assume, without loss of generality, that $\Y$  is a sequence space with a natural product structure so that each possible realisation of the stochastic process corresponds to a unique element of $\Y$ and there exists a measurable map $T:\Y\to\Y$, the time evolution map, which can be seen as the passage of one unit of time, so that 
\[
X_{i-1}\circ T =X_{i}, \quad \mbox{for all $i\in\N$}. 
\]
\begin{note}
There is an obvious relation between $T$ and the \emph{shift} map but we avoid that comparison here because we are definitely not reduced to the usual shift dynamics, in the sense that normally the shift map acts on sequences from a finite or countable alphabet, while here $T$, acts on spaces like $\R^\N$, in the sense that the sequences can be thought as being obtained from an alphabet like $\R$.
\end{note}

Stationarity means that $\p$ is $T$-invariant. Note that $X_i=X_0\circ T^i$, for all $i\in\N_0$, where $T^i$ denotes the $i$-fold composition of $T$, with the convention that $T^0$ denotes the identity map on $\Y$.  
 
We denote by $F$ the cumulative distribution function (d.f.) of $X_0$, \ie $F(x)=\p(X_0\leq x)$. Given any d.f.\ $F$, let $\bar{F}=1-F$ and let $u_F$ denote the right endpoint of the d.f.\ $F$, \ie
$
u_F=\sup\{x: F(x)<1\}.
$ We say we have an \emph{exceedance} of the threshold $u<u_F$ at time $j\in\N_0$ whenever $\{X_j>u\}$
occurs.

We define a new sequence of random variables  $M_1, M_2,\ldots$ given by
\begin{equation}
\label{eq:Mn-definition}
M_n=\max\{X_0,\ldots,X_{n-1}\}.
\end{equation}

We say that we have an \emph{Extreme Value Law} (EVL) for $M_n$ if there is a non-degenerate d.f.\ $H:\R\to[0,1]$ with $H(0)=0$ and,  for every $\tau>0$, there exists a sequence of levels $u_n=u_n(\tau)$, $n=1,2,\ldots$,  such that
\begin{equation}
\label{eq:un}
  n\p(X_0>u_n)\to \tau,\;\mbox{ as $n\to\infty$,}
\end{equation}
and for which the following holds:
\begin{equation}
\label{eq:EVL-law}
\p(M_n\leq u_n)\to \bar H(\tau)=1-H(\tau),\;\mbox{ as $n\to\infty$.}
\end{equation}
where the convergence is meant at the continuity points of $H(\tau)$.

Now let us assume that our underlying system is an ergodic measure-preserving dynamical system $f:\X\to \X$ where $(\X, \B, \mu)$ is a probability space.  (Note that, also, in this context we will often use $\p$ in place of $\mu$ to be more consistent with the rest of the paper.) Given an observable $\varphi:\X\to \R\cup\{\pm\infty\}$, we can define $X_n=\varphi\circ f^n$ for each $n\in \N$ (see Section~\ref{ssec:geom} for more details) and ask what the EVL here is.  If $\varphi$ is sufficiently regular and takes its maximum at a unique point $\zeta\in\X$ then the EVL here is related to the following concept. 

Consider a set $A\in\B$. We define a function that we refer to as \emph{first hitting time function} to $A$, denoted by $r_A:\X\to\N\cup\{+\infty\}$ where
\begin{equation}
\label{eq:hitting-time}
r_A(x)=\min\left\{j\in\N\cup\{+\infty\}:\; f^j(x)\in A\right\}.
\end{equation}
The restriction of $r_A$ to $A$ is called the \emph{first return time function} to $A$. We define the \emph{first return time} to $A$, which we denote by $R(A)$, as the infimum of the return time function to $A$, \ie
\begin{equation}
\label{eq:first-return}
R(A)=\inf_{x\in A} r_A(x).
\end{equation}

Given a point $\zeta\in X$, by Kac Lemma the expected value of $r_{B_\eps(\zeta)}$ when restricted to the $\eps$-ball $B_\eps(\zeta)$ is $1/ \p(B_\eps(\zeta))$.  We say that the system has HTS $H$ for balls around $\zeta$ if for each $\tau\in [0, \infty)$,
$$\lim_{\eps\to 0} \p\left(\left\{\p(B_\eps(\zeta)) r_{B_\eps(\zeta)}>\tau\right\}\right) =H(\tau)$$
for some d.f.\ $H:[0,\infty)\to [0,1]$.

In \cite{FFT10} the link between HTS and EVL was demonstrated and exploited.   So the main question in this paper is what are the orders of 
$$\left|\p(M_n\leq u_n)- \bar H(\tau)\right|$$
and
$$\left|\p\left(\left\{\p(B_\eps(\zeta)) r_{B_\eps(\zeta)}>\tau\right\}\right) -H(\tau)\right|?$$

For the rest of this introduction we will refer to $\tau$ as the time scale; and $n$ and $\p(B_\eps(\zeta))$ as the length scale (although, strictly speaking, the true time is rescaled by the inverse of the length).

In the context of HTS, the early results were mostly for sequences of shrinking dynamically defined cylinders, that is sets $A_n$ where $A_n$ was a maximal subset on which $f^n:A_n\to \X$ is a homeomorphism and $\zeta\in A_n$ for all $n\in \N$.  Moreover, certain strong mixing conditions were imposed.  For example \cite{GS97} gave error terms in terms of a power of the length scale for $\psi$-mixing systems.  A similar result was obtained, but in the context of $\phi$-mixing systems, in \cite{A01}, where lower bounds were also found.  Various similar results are described in \cite{AG01}, including a description of results of \cite{A82} for Markov chains and the results in \cite{HSV99}. 

A major breakthrough was made in \cite{A04} where for $\phi$-mixing systems with cylinders, the error terms also incorporated the time scale, so that increasing time $t$ meant that the error terms decreased by a factor comparable to $H(t)$.  Further refinements were also made in \cite{AV09}.  A similar approach was  used for $\alpha$-mixing processes in \cite{AS11}, but the error terms were not as powerful, in particular, the time scale was not so nicely decoupled from length in the estimates.  We remark that in all the results mentioned so far there is a parameter $\xi_{A_n}$ present, which converges in the limit, but allows a convenient perturbation of the asymptotic law $H$ to improve the apparent convergence. The results we present here are of a similar form to  \cite{AV09}, but are not restricted to balls and do not include this extra factor.

In the dependent EVL context, \cite{MS01} obtained error terms of the order of the length, but not independently of $\tau$.  This is also the case for many of the results in the EVL literature in the i.i.d.\ case.  We should also mention the rather complete results in \cite{K12} (see also \cite{KL09}) where, under an assumption of the nice behaviour of a transfer operator, similar results to ours were proved and applied in the EVL and HTS context (see further comments on this in Section~\ref{subsec:main-results}, in particular Remark~\ref{rem:error-terms}).  One of the main results there, and one of our key motivations here, is to use the error terms to find limit laws for escape of mass through a hole (a small ball), see for example \cite{DY06}.  By decoupling the time parameter in a suitable way in our error terms we are able to estimate how the escape rate depends on the size of the hole.

After finishing this paper, we were made aware of \cite{HN14} which is also concerned with convergence of the observations of maxima from dynamical systems to the relevant EVL.  The error terms there are only given in terms of time $n$, thus fitting with much of the EVL literature.  However, the results, which are applied to a range of one-dimensional systems, are rather fine, principally due to the fact that they make use of the sets $E_n$, described in \cite[Section 2]{C01} and further studied in \cite{HNT12}, which are sets of points which `recur too fast' for the dynamical system in question.  Note that in the latter paper, multiple maxima were also considered.

\subsection{Structure of the paper}
In Section~\ref{sec:gen thm} we first go into a few more technicalities and history of this topic, in particular explaining quantities which deal with clustering and the extremal index, before going on to state and prove our first main result, Theorem~\ref{thm:error-terms-general-SP-no-clustering}.  In Section~\ref{sec:L1 and struc} we suppose that our data comes from a dynamical system with a nice mixing property and state and comment on our other main results, Theorems~\ref{thm:sharp-error-EVL} and \ref{thm:sharp-error-HTS} and Corollary~\ref{cor:zero hole}.  In Section~\ref{sec:Lemmata} we prove some preparatory results we'll use to prove these theorems, while in Section~\ref{sec:EVL err} we prove Theorem~\ref{thm:sharp-error-EVL} and in Section~\ref{sec:HTS err} we prove Theorem~\ref{thm:sharp-error-HTS}.  In Section~\ref{sec:Rych} we give a natural application of our main results.

\textit{Acknowledgements.}   MT would like to thank A.\ Galves for an inspiring conversation. The authors would like to thank H.\ Bruin, N.\ Haydn and G.\ Keller for helpful comments.

\section{Error terms for general stationary stochastic processes under conditions of the type $D_2$ and $D'$}
\label{sec:gen thm}

In what follows for every $A\in\mathcal B$, we denote the complement of $A$ as $A^c:=\mathcal X\setminus A$.

For some $u\in\R$, $q\in \N$, we define the events:
\begin{equation}
\label{eq:U-A-def}
U(u):= \{X_0>u\}\mbox{ and }\A(u):=U(u)\cap\bigcap_{i=1}^{q}T^{-i}(U(u)^c)=\{X_0>u, X_1\leq u, \ldots, X_q\leq u\}.
\end{equation}
We also set $A^{(0)}(u):=U(u)$, $U_n:=U(u_n)$ and $\A_n:=\A(u_n)$, for all $n\in\N$ and $q\in\N_0$.
Let 
\begin{equation}
\label{def:thetan}
\theta_n:=\frac{\p\l\A_n\r}{\p(U_n)}.
\end{equation}

Let $B\in\B$ be an event. For some $s\geq0$ and $\ell\geq 0$, we define:
\begin{equation}
\label{eq:W-def}
\mathscr W_{s,\ell}(B)=\bigcap_{i=\lfloor s\rfloor}^{\lfloor s\rfloor+\max\{\lfloor\ell\rfloor-1,\ 0\}} T^{-i}(B^c).
\end{equation}
The notation $T^{-i}$ is used for the preimage by $T^i$. We will write $\mathscr W_{s,\ell}^c(B):=(\mathscr W_{s,\ell}(B))^c$. 
Whenever is clear or unimportant which event $B\in\B$ applies, we will drop the $B$ and write just $\mathscr W_{s,\ell}$ or $\mathscr W_{s,\ell}^c$. 
Observe that 
\begin{equation}
\label{eq:EVL-HTS}
\mathscr W_{0,n}(U(u))=\{M_n\leq u\}\qquad \mbox{and}\qquad T^{-1}(\mathscr W_{0,n}(B))=\{r_{B}>n\}.
\end{equation}
Also observe that $\mathscr W_{0,n}$ has the following interpretation in terms of the sequence $X_0, X_1, X_2, \ldots$, namely $T^i(x)\notin A_n$ means that $X_i(x)\leq u_n$ or $X_i(x)+j>u_n$ for some $j=1, \ldots, q$.

After the success of the classical  Extremal Types Theorem of Fisher-Tippet and Gnedenko in the i.i.d.\ setting, there has been a great deal of interest in studying the existence of EVL for dependent stationary stochastic processes. Building up on the work of Loynes and Watson, in \cite{L73}, Leadbetter proposed two conditions on the dependence structure of the stochastic processes, which he called $D(u_n)$ and $D'(u_n)$, that guaranteed the existence of the same EVLs of the i.i.d.\ applied to the partial maxima of sequences of random variables satisfying those conditions.

Condition $D(u_n)$ is a sort of uniform mixing condition adapted to this setting of extreme values where the main events of interest are exceedances of the threshold $u_n$. 
Let $F_{i_1,\ldots,i_n}$
denote the joint d.f. of $X_{i_1},\ldots,X_{i_n}$, and set
$F_{i_1,\ldots,i_n}(u)=F_{i_1,\ldots,i_n}(u,\ldots,u)$.
\begin{condition}[$D(u_n)$]We say that $D(u_n)$ holds
for the sequence $X_0,X_1,\ldots$ if for any integers
$i_1<\ldots<i_p$ and $j_1<\ldots<j_k$ for which $j_1-i_p>m$, and any
large $n\in\N$,
\[
\left|F_{i_1,\ldots,i_p,j_1,\ldots,j_k}(u_n)-F_{i_1,\ldots,i_p}(u_n)
F_{j_1,\ldots,j_k}(u_n)\right|\leq \alpha(n,t),
\]
uniformly for every $p,k\in\N$, where $\alpha(n,t_n)\xrightarrow[n\to\infty]{}0$, for some sequence
$t_n=o(n)$.
\end{condition}
Condition $D'(u_n)$ precludes the existence of clusters of exceedances of $u_n$. Let $(k_n)_{n\in\N}$ be a sequence of integers such that
\begin{equation}
\label{eq:kn-sequence-1}
k_n\to\infty,\quad \lim_{n\to\infty}k_n\alpha(n,t_n)=0,\quad\mbox{and}\quad  k_n t_n = o(n).
\end{equation}
\begin{condition}[$D'(u_n)$]\label{cond:D'} We say that $D'(u_n)$
holds for the sequence $X_0$, $X_1$, $X_2$, $\ldots$ if there exists a sequence $\{k_n\}_{n\in\N}$ satisfying \eqref{eq:kn-sequence-1} and such that
\begin{equation*}
\lim_{n\rightarrow\infty}\,n\sum_{j=1}^{\lfloor n/k_n \rfloor}\p( X_0>u_n,X_j>u_n)=0.
\end{equation*}
\end{condition}
Under these two conditions the EVL obtained corresponds to a standard exponential distribution, where $\bar H(\tau)=\e^{-\tau}$.

In certain circumstances, observed data clearly showed the existence of clusters of exceedances, which meant that $D'(u_n)$ did not hold. This motivated the study of EVL and the affect of clustering.  In fact it was observed that clustering of exceedances essentially produced the same type of EVL but with a parameter $0\leq\theta\leq1$, the Extremal Index (EI), so that $\bar H(\tau)=\e^{-\theta\tau}$: here $\theta$ quantifies the intensity of clustering. In order to show the existence of EVLs with a certain EI $\theta\leq1$, new conditions (replacing $D'(u_n)$) were devised. We mention condition $D''(u_n)$ of  \cite{LN89} and particularly the more general condition $D^{(k)}(u_n)$ of \cite{CHM91}, which also includes the case of absence of clustering. In fact, $D^{(k)}(u_n)$ (in the formulation of \cite[Equation (1.2)]{CHM91}) is equal to $D'(u_n)$, when $k=1$ and to $D''(u_n)$ when $k=2$. Together with condition $D(u_n)$, the  condition  $D^{(k)}(u_n)$) gave an EVL, where $\bar H(\tau)=\e^{-\theta\tau}$, with an EI $\theta$ given by O'Brien's formula, whenever the following limit exists:
\begin{equation}
\label{eq:OBrien-EI}
\theta=\lim_{n\to\infty}\theta_n=\lim_{n\to\infty}\frac{\p\l\A_n\r}{\p(U_n)}.
\end{equation}

When the stochastic processes arise from dynamical systems as described in Section~\ref{sec:L1 and struc} below, condition $D(u_n)$ cannot be verified using the usual available information about mixing rates of the system except in some very special situations, and even then only for certain subsequences of $n$. This means that the theory developed by Leadbetter and others (such as Nandagopalan, Chernick, Hsing and McCormick) is not practical in this dynamical systems context. For that reason, motivated by the work of Collet (\cite{C01}), the first and second named authors proposed a new condition called $D_2(u_n)$ for general stationary stochastic processes, which imposes a much weaker uniformity requirement than $D(u_n)$, which together with $D'(u_n)$ admitted a proof of the existence of EVL in the absence of clustering (with $\theta=1$). The great advantage of $D_2(u_n)$ is that it is so much weaker than $D(u_n)$, in what respects to uniformity, and follows easily for stochastic processes arising from systems with sufficiently fast decay of correlations. In the argument of \cite{FF08}, this weakening was achieved by a fuller application of condition $D'(u_n)$. 

In \cite{FFT12}, the authors proved a connection between periodicity and clustering. Motivated by the behaviour at periodic points, which created the appearance of clusters of exceedances, the authors proposed new conditions in order to prove the existence of EVLs with EI less than 1. The main idea was that, under a condition $S\!P_{p,\theta}(u_n)$, which imposed some sort of periodic behaviour of period $p$ on the structure of general stationary stochastic processes, the sequences $\p(M_n\leq u_n)$ and  $\p(\mathscr W_{0,n}(A^{(p)}(u_n))$ share the same limit (see \cite[Proposition~1]{FFT12}). Then the strategy was to prove the existence of a limit for $\p(\mathscr W_{0,n}(A^{(p)}(u_n))$, which was achieved under conditions $D^p(u_n)$ and $D'_p(u_n)$. These latter conditions can be seen as being obtained from $D_2(u_n)$ from \cite{FF08} and the original $D'(u_n)$, respectively, by replacing the role of exceedances $\{X_j>u_n\}$ by that of \emph{escapes}, which correspond to the event $\{X_j>u_n, X_{j+p}\leq u_n\}$. 

In \cite{AFV12} discontinuity points create two periodic types of behaviour (with possibly different periods) on the structure of the stochastic processes, so some further adjustments to conditions $D_p(u_n)$ and $D'_p(u_n)$ were needed. 

We remark that in all the cases above the main advantage of the conditions
$D_2(u_n)$, $D^p(u_n)$ is that they are much weaker than the original uniformity requirement imposed by $D(u_n)$ and, unlike $D(u_n)$, they all follow from sufficiently fast decay of correlations of the system.

While developing the techniques in this paper to sharpen the error terms, as in Sections~\ref{sec:EVL err} and \ref{sec:HTS err}, it was necessary to improve the estimates in \cite[Proposition~1]{FFT12} (this is done in Proposition~\ref{prop:relation-balls-annuli-general} below). One consequence of this is that we were then able to essentially remove condition $S\!P_{p,\theta}(u_n)$. Whence we refine all the conditions to obtain EVL in order to, on one hand, obtain a unified statement of the conditions which includes simultaneously the cases of absence and presence of clustering and, on the other hand, to combine all the scenarios considered before with no periodic behaviour, with simple periodic behaviour or multiple types of periodic behaviour. 

As can be seen in the historical discussion above, the conditions $D$ often come with many subscripts and superscripts.  To simplify the notation, here we employ instead a cyrillic D, i.e., $\D$.

\begin{condition}[$\D(u_n)$]\label{cond:D} We say that $\D(u_n)$ holds for the sequence $X_0,X_1,\ldots$ if for every  $\ell,t,n\in\N$ and $q\in\N_0$,
\begin{equation}\label{eq:D1}
\left|\p\left(\A_n\cap
 \mathscr W_{t,\ell}\left(\A_n\right) \right)-\p\left(\A_n\right)
  \p\left(\mathscr W_{0,\ell}\left(\A_n\right)\right)\right|\leq \gamma(q,n,t),
\end{equation}
where $\gamma(q,n,t)$ is decreasing in $t$ for each $q, n$ and, for every $q\in\N_0$, there exists a sequence $(t_n)_{n\in\N}$ such that $t_n=o(n)$ and
$n\gamma(q,n,t_n)\to0$ when $n\rightarrow\infty$.
\end{condition}

\begin{remark}
\label{rem:D}
Note that, the new condition $\D(u_n)$ does not impose any uniform bound independent of the number $q$ of random variables considered in the first block. On the contrary, the original $D(u_n)$ imposed a uniform bound independent of the number $p$ of random variables considered in the first block. This is the crucial difference that makes it possible to prove $\D(u_n)$ easily from decay of correlations of the underlying stochastic processes, in contrast to $D(u_n)$, which  is not possible to verify even in very simple situations. See Remark~\ref{rem:D-dc} for further explanations. The weakening of the uniformity imposed by $D(u_n)$ came at price on the rate function: while for, $D(u_n)$, we need $\lim_{n\to\infty}k_n\alpha(n,t_n)=0$, for $\D(u_n)$, we need $\lim_{n\to\infty}n\gamma(q,n,t_n)=0$. However, this is a very small price to pay since, when the stochastic processes arise from dynamical systems, to verify $\D(u_n)$ means we need decay of correlations at least at a summable rate. But this is precisely the regime at which one can prove Central Limit Theorems. Hence, even though we have a slight strengthening of the mixing rate, when we compare $\D(u_n)$ to $D(u_n)$, the weakening on the uniformity is so much more important that we believe it is fair to say  $\D(u_n)$ is considerably weaker than $D(u_n)$. This is cemented by the fact that $\D(u_n)$  can be verified in a huge range of examples arising from dynamical systems, where condition $D(u_n)$ simply cannot.
\end{remark}

For some fixed $q\in\N_0$, consider the sequence $(t_n)_{n\in\N}$, given by condition  $\D(u_n)$ and let $(k_n)_{n\in\N}$ be another sequence of integers such that 
\begin{equation}
\label{eq:kn-sequence}
k_n\to\infty\quad \mbox{and}\quad  k_n t_n = o(n).
\end{equation}

\begin{condition}[$\D'_q(u_n)$]\label{cond:D'q} We say that $\D'_q(u_n)$
holds for the sequence $X_0,X_1,X_2,\ldots$ if there exists a sequence $(k_n)_{n\in\N}$ satisfying \eqref{eq:kn-sequence} and such that
\begin{equation}
\label{eq:D'rho-un}
\lim_{n\rightarrow\infty}\,n\sum_{j=q+1}^{\lfloor n/k_n\rfloor-1}\p\left( \A_n\cap T^{-j}\left(\A_n\right)
\right)=0.
\end{equation}
\end{condition}

\begin{remark}
Note that condition $\D'_q(u_n)$ is very similar to condition $D^{(q+1)}(u_n)$ from \cite{CHM91}. Since $T^{-j}\left(\A_n\right)\subset \{X_j>u_n\}$, it is slightly weaker than $D^{(q+1)}(u_n)$ in the formulation of \cite[Equation (1.2)]{CHM91} but in the applications considered that does not make any difference. Note that if $q=0$ then we get condition $D'(u_n)$ from Leadbetter. 
\end{remark}

The following is the most general of the main theorems in this paper. Since condition $\D(u_n)$ is much weaker than the original $D(u_n)$ of Leadbetter, in the sense explained in Remark~\ref{rem:D}, then Theorem~\ref{thm:error-terms-general-SP-no-clustering} can be seen, in particular, as a generalisation of \cite[Corollary~1.3]{CHM91}, where the error terms are computed and with a much more extensive potential of application, which can be fully appreciated by the examples of stochastic processes arising from dynamical systems that can be addressed.

\begin{theorem}
\label{thm:error-terms-general-SP-no-clustering}
Let $X_0, X_1, \ldots$ be a stationary stochastic process and $(u_n)_{n\in\N}$ a sequence satisfying \eqref{eq:un}, for some $\tau>0$.  Assume that conditions $\D(u_n)$ and $\D_q'(u_n)$ hold for some $q\in\N_0$, and $(t_n)_{n\in\N}$ and $(k_n)_{n\in\N}$ are the sequences in those conditions.
Then, there exists $C>0$ such that for all $n$ large enough we have \begin{align*}
\big|\p(M_n\leq u_n)-\e^{-\theta_n\tau}\big|&\leq  C\Bigg[k_nt_n\frac{\tau}n+n\gamma(q,n,t_n)+n\sum_{j=q+1}^{\lfloor n/k_n \rfloor-1} \p\l\A_n\cap T^{-j}\l\A_n\r\r\\
&\quad+ \e^{-\theta_n\tau}\left(\left|\tau-n\p\l U_n\r\right|+\frac {\tau^2}{k_n}\right)+q\p\l U_n\setminus \A_n\r\Bigg],
\end{align*}
where $\theta_n$ is given by equation \eqref{def:thetan}.
\end{theorem}

In case the limit in \eqref{eq:OBrien-EI} exists then we can use the previous result to obtain:

\begin{corollary}
\label{cor:error-terms-general-SP-no-clustering}
Let $X_0, X_1, \ldots$ be a stationary stochastic process and $(u_n)_{n\in\N}$ a sequence satisfying \eqref{eq:un}, for some $\tau>0$. Assume that conditions $\D(u_n)$ and $\D_q'(u_n)$ hold for some $q\in\N_0$, and $(t_n)_{n\in\N}$ and $(k_n)_{n\in\N}$ are the sequences in those conditions.  Moreover assume that the limit in \eqref{eq:OBrien-EI} exists.
Then, there exists $C>0$ such that for all $n\in\N$ we have \begin{align*}
\big|\p(M_n\leq u_n)-\e^{-\theta\tau}\big|&\leq  C\Bigg[k_nt_n\frac{\tau}n+n\gamma(q,n,t_n)+n\sum_{j=1}^{\lfloor n/k_n \rfloor} \p\l\A_n\cap T^{-j}\l\A_n\r\r\\
&\quad+ \e^{-\theta\tau}\left(\left|\tau-n\p\l U_n\r\right|+\frac {\tau^2}{k_n}+\left|\theta_n-\theta\right|\tau\right)+q\p\l U_n\setminus \A_n\r\Bigg],
\end{align*}
where $\theta_n$ is given by equation \eqref{def:thetan} and $\theta$ is given by equation \eqref{eq:OBrien-EI}.
\end{corollary}

This applies to general stochastic processes as well as naturally applying to a large class of dynamical systems such as those studied in \cite{FFT10, FFT11, FFT12, HN14}.  Our focus on examples in this paper will be principally concerned with the later, stronger, results.

\begin{remark}
\label{rem:no-D'}
Note that the estimates of Theorem~\ref{thm:error-terms-general-SP-no-clustering} and Corollary~\ref{cor:error-terms-general-SP-no-clustering} hold under $\D(u_n)$ alone. However, if $\D'_q(u_n)$ does not hold, the upper bound is useless since the third term on the right hand side would not converge to $0$ as $n\to\infty$.
\end{remark}

\begin{remark}
\label{rem:rates-general-case}
	Behind the convergence result there is a `blocking argument' and several quantities are related to it. Namely, the number of blocks taken ($k_n$) and the size of the gaps between the blocks ($t_n$) have to satisfy \eqref{eq:kn-sequence}. The first error term, depending on the choices for adequate $k_n$ and $t_n$, typically, decays like $n^{-\delta}$, for some $0<\delta<1$. The second term depends on the long range mixing rates of the process ($\D(u_n)$). The third term takes into account the short range recurrence properties ($\D'(u_n)$). The fourth has three components, the first depends on the asymptotics of relation \eqref{eq:un}, the third on the asymptotics of \eqref{eq:OBrien-EI} and the second on the number of blocks, which must be traded off with the first term. Note that the term $\e^{-\theta\tau}\frac{\tau^2}{k_n}$ also appears in the i.i.d. case since it results from expansion \eqref{eq:exponential-estimate} below. The fifth term results from replacing $U_n$ by $A_n^{(q)}$ (see Proposition~\ref{prop:relation-balls-annuli-general}) and should decay like $1/n$. The constant $C$ may depend on the rates just mentioned but not on $\tau$.
\end{remark}

The rest of this section is devoted to the proof of Theorem~\ref{thm:error-terms-general-SP-no-clustering} and Corollary~\ref{cor:error-terms-general-SP-no-clustering}.

The following result gives a simple estimate but a rather important one. It is crucial in removing condition $S\!P_{p,\theta}(u_n)$ from \cite{FFT12}, to present in a unified way the results under the presence and absence of clustering in Theorem~\ref{thm:error-terms-general-SP-no-clustering} and, most of all, to obtain the sharper results in Theorems~\ref{thm:sharp-error-EVL} and \ref{thm:sharp-error-HTS}.   

\begin{proposition}
\label{prop:relation-balls-annuli-general}
Given an event $B\in\mathcal B$, let $q,n\in\N$ be such that $q<n$ and define $A=B\setminus \bigcup_{j=1}^{q} T^{-j}(B)$. Then
$$
\left|\p(\mathscr W_{0,n}(B))-\p(\mathscr W_{0,n}(A))\right|\leq \sum_{j=1}^{q} \p\left(\mathscr W_{0,n}(A)\cap T^{-n+j}(B\setminus A)\right).
$$
\end{proposition}

\begin{proof}
Since $A\subset B$, then clearly $\mathscr W_{0,n}(B)\subset \mathscr W_{0,n} (A)$. Hence, we have to estimate the probability of $\mathscr W_{0,n} (A)\setminus \mathscr W_{0,n}(B)$ which corresponds to the set of points that at some time before $n$ enter the $B$ but never enter its subset $A$. 

Let $x\in\mathscr W_{0,n} (A)\setminus \mathscr W_{0,n}(B)$. Then  $T^j(x)\in B$ for some  $j=1,\ldots,n-1$ but $T^j(x)\notin A$ for all such $j$. We will see that there exists $j\in\{1,\ldots, q\}$ such that $T^{n-j}(x)\in B$. In fact, suppose that no such $j$ exists. Then let $\ell=\max\{i\in\{1,\ldots, n-1\}:\, T^{i}(x)\in B\}$ be the last moment the orbit of $x$ enters $B$ during the time period in question. Then, clearly, $\ell<n-q$. Hence, if $T^i(x)\notin B$, for all $i=\ell+1,\ldots,n-1$ then we must have that $T^\ell(x)\in A$ by definition of $A$. But this contradicts the fact that $x\in\mathscr W_{0,n} (A)$. Consequently, we have that there exists $j\in\{1,\ldots, q\}$ such that $T^{n-j}(x)\in B$ and since $x\in\mathscr W_{0,n} (A)$ then we can actually write $T^{n-j}(x)\in B\setminus A$.  

This means that $\mathscr W_{0,n} (A)\setminus \mathscr W_{0,n}(B)\subset \bigcup_{j=1}^q T^{-n+j}(B\setminus A)\cap \mathscr W_{0,n} (A)$ and then
\begin{multline*}
\big|\p(\mathscr W_{0,n}(B)-\p(\mathscr W_{0,n}(A))\big|=\p(\mathscr W_{0,n} (A)\setminus \mathscr W_{0,n}(B))\\ \leq 
\p\left(\bigcup_{j=1}^q T^{-n+j}(B\setminus A)\cap \mathscr W_{0,n} (A)\right)\leq\sum_{j=1}^{q} \p\left(\mathscr W_{0,n}(A)\cap T^{-n+j}(B\setminus A)\right),
\end{multline*}
as required.
\end{proof}

In what follows we will need the error term of the limit expression $\lim_{k\to\infty}\left(1+\frac x n\right)^n=\e^x$, namely,\begin{equation}
\label{eq:exponential-estimate}
\left(1+\frac x n\right)^n=\e^x\left(1-\frac {x^2}{2n}+\frac{x^3(8+3x)}{24n^2}+O\left(\frac1{n^3}\right)\right),
\end{equation}
which holds uniformly for $x$ on bounded sets. Also, by Taylor's expansion,  for every $\delta\in\R$ and $x\in\R$ we have
\begin{equation}
\label{eq:exponential-estimate2}
\left|\e^{x+\delta}-\e^{x}\right|\leq \e^{x}\left(|\delta|+\e^{|\delta|}\delta^2/2\right).
\end{equation}

The strategy is to use a blocking argument, that goes back to Markov, which consists of splitting the data into blocks with gaps of increasing length. There are three main steps. The first step is to estimate the error produced by neglecting the data corresponding to the gaps. The second is to use essentially the mixing condition $\D(u_n)$ (with some help from $\D'_q(u_n)$) to show that the probability of the event corresponding to the global maximum being less than some threshold $u_n$ can be approximated by the product of the probabilities of the maxima within each block being less than $u_n$. The idea is that the gaps make the maxima in each block become practically independent from each other. The last step is to use condition $\D'_q(u_n)$ to estimate the probability of the maximum within a block being smaller than $u_n$.   

Next, we state a couple of lemmas and a proposition that give the main estimates regarding the use of a blocking argument. Their proofs can be found in \cite[Section~3.4]{F13}. 

\begin{lemma}
\label{lem:time-gap-1}
For any fixed $A\in \mathcal B$ and $s,t',m\in\N$, 
 we have:
\begin{equation*}
\left|\p(\mathscr W_{0,s+t'+m}(A))-\p(\mathscr W_{0,s}(A)\cap \mathscr W_{s+t',m}(A))\right|\leq t'\p(A).
\end{equation*}
 \end{lemma}

\begin{lemma}
\label{lem:inductive-step-1}For any fixed $A\in \mathcal B$ and integers $s,t,m$, we have:
\begin{multline*}
\left|\p(\mathscr W_{0,s}(A)\cap \mathscr W_{s+t,m}(A))-\p(\mathscr W_{0,m}(A))(1-s\p(A))\right|\leq\\ \left|s\p(A)\p(\mathscr W_{0,m}(A))-\sum_{j=0}^{s-1}\p(A\cap \mathscr  W_{s+t-j,m}(A))\right| +2s\sum_{j=1}^{s-1}\p(A\cap T^{-j}(A)).
\end{multline*}
 \end{lemma}

\begin{propositionP}
\label{prop:main-estimate-1} Fix $A\in \mathcal B$ and $n\in\N$. Let $\ell,k,t\in\N$ be such that $\ell=\lfloor n/k \rfloor$ and $\ell\p(A)<1$.  We have:
\begin{multline*}
\left|\p(\mathscr W_{0,n}(A))-(1-\ell\,\p(A))^{k}\right|\leq \\ 2kt\p(A)+ 2n\sum_{j=1}^{\ell-1}\p(A\cap T^{-j}(A))
+\sum_{i=0}^{k-1} \left|\ell\p(A)\p(\mathscr W_{0,i(\ell+t)})-\sum_{j=0}^{\ell-1}\p(A\cap \mathscr  W_{\ell+t-j,i(\ell+t)})\right|. \end{multline*}
 \end{propositionP}

\begin{proof}
 The basic idea is to split the time interval $[0,n)$ into $k$ blocks of size $\lfloor n/k \rfloor$. Then, using Lemma~\ref{lem:time-gap-1} we add gaps of size $t$ between the blocks, and next we apply Lemma~\ref{lem:inductive-step-1} recursively until we exhaust all the blocks. 

Noting that $0\leq k(\ell+t)-n\leq kt$ and using Lemma~\ref{lem:time-gap-1}, with $s=n$, $t'= k(\ell+t)-n$, $m=0$, and setting $\mathscr W_{i,0}:=\X$, for all $i=0,1,2,\ldots$ as well as $\mathscr W_{i,n}=\mathscr W_{i,n}(A)$ for $n\in \N$, we have:
\begin{equation}
\label{eq:approx1-1}
\Big|\p(\mathscr W_{0,n})-\p(\mathscr W_{0,k(\ell+t)})\Big|\leq kt\p(A). 
\end{equation}

  It follows by using Lemmas~\ref{lem:time-gap-1} and \ref{lem:inductive-step-1} that
\begin{align}
\Big|\p(\mathscr W_{0,i(\ell+t)})-&(1-\ell\p(A))\p(\mathscr W_{0,(i-1)(\ell+t)})\Big|\leq 
\left|\p(\mathscr W_{0,i(\ell+t)})-\p(\mathscr W_{0,\ell}\cap\mathscr W_{(\ell+t),(i-1)(\ell+t)})\right|\nonumber\\
&\quad+\left|\p(\mathscr W_{0,\ell}\cap\mathscr W_{(\ell+t),(i-1)(\ell+t)})-(1-\ell\p(A))\p(\mathscr W_{0,(i-1)(\ell+t)})\right|\nonumber \\
&\leq t\p(A)+\left|\ell\p(A)\p(\mathscr W_{0,(i-1)(\ell+t)})-\sum_{j=0}^{\ell-1}\p(A\cap \mathscr  W_{\ell+t-j,(i-1)(\ell+t)})\right| \nonumber \\
&\quad+ 2\ell\sum_{j=1}^{\ell-1}\p(A\cap T^{-j}(A)),\label{eq:induction-step-1}
\end{align} 
Let $\Upsilon_i:=t\p(A)+\left|\ell\p(A)\p(\mathscr W_{0,i(\ell+t)})-\sum_{j=0}^{\ell-1}\p(A\cap \mathscr  W_{\ell+t-j,i(\ell+t)})\right|+2\ell\sum_{j=1}^{\ell-1}\p(A\cap T^{-j}(A)).$ 
 Since $\ell\p(A)<1$, then it is clear that $|(1-\ell\p(A))|<1$. Also, note that $|\p(\mathscr W_{0,\ell+t})-(1-\ell\p(A))|\leq \Upsilon_0$. Now, we use (\ref{eq:induction-step-1}) recursively to estimate $\left|\p(\mathscr W_{0,k(\ell+t)})-(1-\ell\p(A))^k\right|$. In fact, we have
 \begin{align}
\left|\p(\mathscr W_{0,k(\ell+t)})-(1-\ell\p(A))^k\right|&\leq \sum_{i=0}^{k-1} (1-\ell\p(A))^{k-1-i}\left|\p(\mathscr W_{0,(i+1)(\ell+t)})-(1-\ell\p(A))\,\p(\mathscr W_{0,i(\ell+t)})\right|\nonumber\\ \label{eq:recursive-estimate-1}
&\leq \sum_{i=0}^{k-1} (1-\ell\p(A))^{k-1-i}\,\Upsilon_i \leq \sum_{i=0}^{k-1}\Upsilon_i
\end{align}
The result follows now at once from (\ref{eq:approx1-1}) and (\ref{eq:recursive-estimate-1}).
\end{proof}

We are now in a position to prove Theorem~\ref{thm:error-terms-general-SP-no-clustering}.

\begin{proof}[Proof of Theorem~\ref{thm:error-terms-general-SP-no-clustering}]
Letting $A=\A_n$, $\ell=\lfloor n/k_n\rfloor$, $k=k_n$ and $t=t_n$ on Proposition~\ref{prop:main-estimate-1}, we obtain
\begin{multline}
\label{eq:error1}
\left|\p(\mathscr W_{0,n}(\A_n))-\left(1-\left\lfloor \frac n{k_n}\right\rfloor\p(\A_n)\right)^{k_n}\right|\leq  2k_nt_n\p(U_n)+ 2n\sum_{j=1}^{\lfloor n/k_n\rfloor-1}\p\l\A_n \cap T^{-j}\A_n\r\\
+\sum_{i=0}^{k_n-1}\left|\left\lfloor \frac n{k_n}\right\rfloor\p(\A_n)\p\left(\mathscr W_{0,i(\ell_n+t_n)}\l\A_n\r\right)-\sum_{j=0}^{\lfloor n/k_n\rfloor-1}\p\left(\A_n\cap \mathscr  W_{\ell_n+t_n-j,i(\ell_n+t_n)}\l\A_n\r\right)\right|.
\end{multline}
Using condition $\D(u_n)$, we have that for the third term:
\begin{equation}
\label{eq:error1.5}
\sum_{i=0}^{k_n-1}\left|\left\lfloor \frac n{k_n}\right\rfloor\p(\A_n)\p\left(\mathscr W_{0,i(\ell+t)}\l\A_n\r\right)-\sum_{j=0}^{\lfloor n/k_n\rfloor-1}\p\left(\A_n\cap \mathscr  W_{\ell+t-j,i(\ell+t)}\l\A_n\r\right)\right|\leq n\gamma(q,n,t_n).
\end{equation}  

By \eqref{eq:exponential-estimate2}, we have that there exists $C$ such that
\begin{align*}
\left|\e^{-\left\lfloor \frac n{k_n}\right\rfloor k_n\p\l\A_n\r}-\e^{-\theta_n \tau}\right|&\leq \e^{-\theta_n\tau}\left[\left|\theta_n\tau-\left\lfloor \frac n{k_n}\right\rfloor k_n\p\l\A_n\r\right|+\o\left(\left|\theta_n\tau-\left\lfloor \frac n{k_n}\right\rfloor k_n\p\l\A_n\r\right|\right)\right]\nonumber\\
&\leq C   \e^{-\theta_n\tau}\left|\theta_n\tau-n\p\l\A_n\r\right|\nonumber\\
&\leq C   \e^{-\theta_n\tau}\left|\tau-n\p\l U_n\r\right|.
\end{align*}
Using \eqref{eq:exponential-estimate} and \eqref{eq:exponential-estimate2}, there exists  $C'>0$  such that
\begin{align*}
\left|\left(1-\left\lfloor \frac n{k_n}\right\rfloor\p\left(\A_n\right)\right)^{k_n}-\e^{-\left\lfloor \frac n{k_n}\right\rfloor k_n\p\l\A_n\r}\right|&=\e^{-\left\lfloor \frac n{k_n}\right\rfloor k_n\p\l\A_n\r}\left(\frac{\left(n\p\l\A_n\r\right)^2}{2k_n}+\o\left(\frac{1}{k_n}\right)\right)\nonumber \\
&\leq \e^{-\theta_n\tau} \left(\frac {\tau^2}{k_n}+C\frac {\tau^2}{k_n}\left|\tau-n\p\l U_n\r\right|+\o\left(\frac{1}{k_n}\right)\right)
\\
&\leq C'\e^{-\theta_n\tau} \frac {\tau^2}{k_n}.
\end{align*}
Hence, there exists $C''>0$, depending on $\n$ but not on $\tau$, such that
\begin{equation}
\label{eq:error2}
\left|\left(1-\left\lfloor \frac n{k_n}\right\rfloor\p\left(\A_n\right)\right)^{k_n}-\e^{-\theta_n\tau}\right|\leq C'' \e^{-\theta_n\tau}\left(\left|\tau-n\p\l U_n\r\right|+\frac {\tau^2}{k_n}\right).
\end{equation}

Finally, by Proposition~\ref{prop:relation-balls-annuli-general}  we have
\begin{align}
\label{eq:error3}
\left|\p(M_n\leq u_n)-\p\left(\mathscr W_{0,n}\left(\A_n\right)\right)\right|&\leq\sum_{j=1}^{q} \p\left(\mathscr W_{0,n}\left(\A_n\right)\cap T^{-n+j}(U_n\setminus \A_n)\right)\nonumber\\
&\leq q\p\l U_n\setminus \A_n\r.
\end{align}
Note that when $q=0$ both sides of inequality \eqref{eq:error3} equal 0. 

The estimate in Theorem~\ref{thm:error-terms-general-SP-no-clustering} follows from joining the estimates in \eqref{eq:error1}, \eqref{eq:error1.5}, \eqref{eq:error2} and \eqref{eq:error3}. 
\end{proof}

\begin{proof}[Proof of Corollary~\ref{cor:error-terms-general-SP-no-clustering}]
By \eqref{eq:exponential-estimate2}, we have that there exists $C>0$ such that
\begin{align}
\left|\e^{-\theta_n \tau}-\e^{-\theta \tau}\right|&\leq \e^{-\theta\tau}\left[|\theta_n-\theta|\tau+o(|\theta_n-\theta|)\right]\nonumber\\
&\leq C   \e^{-\theta\tau}|\theta_n-\theta|\tau \label{eq:teta_tetan}.
\end{align}

By Theorem~\ref{thm:error-terms-general-SP-no-clustering} and (\ref{eq:teta_tetan}), there exists $C'>0$ such that 
 \begin{align*}
\nonumber \big|\p(M_n\leq u_n)-\e^{-\theta_n\tau}\big|&\leq  C'\Bigg(k_nt_n\frac{\tau}n+n\gamma(q,n,t_n)+n\sum_{j=1}^{\lfloor n/k_n \rfloor-1} \p\l\A_n\cap T^{-j}\l\A_n\r\r\\
&\quad+ \e^{-\theta\tau}(1+C  |\theta_n-\theta|\tau)\left(\left|\tau-n\p\l U_n\r\right|+\frac {\tau^2}{k_n}\right)+q\p\l U_n\setminus \A_n\r\Bigg).
\end{align*}

So, there exists $C''>0$ such that
 \begin{align}
\nonumber \big|\p(M_n\leq u_n)-\e^{-\theta_n\tau}\big|&\leq  C''\Bigg(k_nt_n\frac{\tau}n+n\gamma(q,n,t_n)+n\sum_{j=1}^{\lfloor n/k_n \rfloor-1} \p\l\A_n\cap T^{-j}\l\A_n\r\r\\
&\qquad \quad+ \e^{-\theta\tau}\left(\left|\tau-n\p\l U_n\r\right|+\frac {\tau^2}{k_n}\right)+q\p\l U_n\setminus \A_n\r\Bigg). \label{eq:interm}
\end{align}

The result follows from   (\ref{eq:teta_tetan}) and  (\ref{eq:interm}).

\end{proof}

\section{Sharper estimates under stronger assumptions on the system}
\label{sec:L1 and struc}

In this section we make stronger assumptions on our system in order to obtain better error estimates.  Firstly we make an important mixing assumption and secondly we ask for a bit more geometric structure around our points of interest. 

Take a system  $(\X,\mathcal
B,\p,f)$, where $\X$ is a Riemannian manifold, $\mathcal B$ is the Borel sigma-algebra,  $f:\X\to\X$ is a measurable map and $\p$ an $f$-invariant probability measure.  

\subsection{Decay of correlations against $L^1$}

 Our basic assumption will be decay of correlations against $L^1$ observables. The definition of this is as follows.
Let \( \mathcal C_{1}, \mathcal C_{2} \) denote Banach spaces of real valued measurable functions defined on \( \X \).
We denote the \emph{correlation} of non-zero functions $\phi\in \mathcal C_{1}$ and  \( \psi\in \mathcal C_{2} \) w.r.t.\ a measure $\p$ as
\[
\cv_\p(\phi,\psi,n):=\frac{1}{\|\phi\|_{\mathcal C_{1}}\|\psi\|_{\mathcal C_{2}}}
\left|\int \phi\, (\psi\circ f^n)\, \dif\p-\int  \phi\, \dif\p\int
\psi\, \dif\p\right|.
\]

We say that we have \emph{decay
of correlations}, w.r.t.\ the measure $\p$, for observables in $\mathcal C_1$ \emph{against}
observables in $\mathcal C_2$, if there exists a non-increasing rate function $\gamma:\N\to \R$, with $\lim_{n\to\infty}\gamma(n)=0$, such that, for every $\phi\in\mathcal C_1$ and every
$\psi\in\mathcal C_2$, we have
 $$\cv_\p(\phi,\psi,n)\leq \gamma(n). 
 $$

We say that we have \emph{decay of correlations against $L^1$
observables} whenever  this holds for $\mathcal C_2=L^1(\p)$  and
$\|\psi\|_{\mathcal C_{2}}=\|\psi\|_1=\int |\psi|\,\dif\p$.

\begin{remark}
\label{rem:decay-L1}
Examples of systems for which we have decay of correlations against $L^1$ observables include: non-uniformly expanding maps of the interval, like the ones considered by Rychlik in \cite{R83}, for which $\mathcal C_1$ is the space of functions of bounded variation (see Section~\ref{sec:Rych} for more details); and non-uniformly expanding maps in higher dimensions, like the ones studied by Saussol in \cite{S00}, for which $\mathcal C_1$ is the space of functions with bounded quasi-H\"older norm. 
\end{remark}

\begin{remark}
We remark that, in most situations, decay of correlations against $L^1$ observables is a consequence of the existence of a gap in the spectrum of the map's corresponding Perron-Frobenius operator. However, in  \cite{D98} Dolgopyat proves exponential decay of correlations for certain Axiom A flows but along the way he proves it for semiflows against $L^1$ observables. This is done via estimates on families of twisted transfer operators for the Poincar\'e map, but without considering the Perron-Frobenius operator for the flow itself. This means that the discretisation of this flow by using a time-1 map provides an example of a system with decay of correlations against $L^1$ for which it is not known if there exists a spectral gap of the corresponding Perron-Frobenius operator. 
As we have said, existence of a spectral gap for the map's Perron-Frobenius operator, defined in some nice function space, appears to be a stronger property than decay of correlations against $L^1$ observables.
However, the latter is still a very strong property. In fact, from decay of correlations against $L^1$ observables, regardless of the rate, as long as it is summable, one can  actually  show that the system has exponential decay of correlations of H\"older observables against $L^\infty$. (See \cite[Theorem~B]{AFLV11}).
\end{remark}

\begin{remark}
The results below assume that the system has decay of correlations against $L^1$ observables. However, we do not need this assumption in full strength since we do not need it to hold for \emph{all} $L^1$ observables. Namely, we only need that for all $\phi\in\mathcal C_1$, where $\mathcal C_1$ is some Banach space of real valued functions, and all $\psi$ of the form $\psi=\I_A$, for some $A\in\mathcal B$, we have that $\gamma(n):=\cv(\phi,\psi,n)$ is such that $n^2\gamma(n)\xrightarrow[n\to\infty]{}0$, where (and this is a crucial point) the $\|\cdot\|_{\mathcal C_2}$ appearing in $\cv(\phi,\psi,n)$ is the $L^1(\p)$-norm, \ie   $\|\psi\|_{\mathcal C_2}=\|\I_A\|_1=\p(A)$.
\end{remark}

\begin{remark}
The basic assumption to be used below is decay of correlations against $L^1$ observables. Under this assumption, as in \cite{AFV12}, conditions $\D(u_n)$ and $\D'(u_n)$ can be shown to hold for observable functions for which conditions \ref{item:U-ball} and \ref{item:repeller} (when $q\neq0$), also described below, hold. To obtain sharper results we need to elaborate further in order to control the error terms better. This means we will not use these conditions below.
\end{remark}

\begin{remark}
\label{rem:D-dc}
One of the most important advantages of condition $\D(u_n)$ when compared to the original condition $D(u_n)$ introduced by Leadbetter is the fact that $\D(u_n)$ follows easily from decay of correlations (not necessarily against $L^1$). In fact, if $\I_{A^{(q)}_n}\in \mathcal C_1$, it follows trivially by taking $\phi= \I_{A^{(q)}_n}$ and $\psi=\I_{\mathscr W_{0,\ell}(A^{(q)}_n)}$. See \cite[Section~5.1]{F13} for a discussion on the subject. This is the case when $\mathcal C_1$ is the space of functions of bounded variation or bounded quasi-H\"older norm used in \cite{S00}. Moreover, even when we consider $\mathcal C_1$ to be the space of H\"older continuous functions, in which case $\I_{A^{(q)}_n}\notin \mathcal C_1$, condition $\D(u_n)$ can still be proved using a suitable H\"older continuous approximation for $\I_{A^{(q)}_n}$, as in  \cite[Proposition~5.2]{F13}. Now, observe that the original condition $D(u_n)$ does not follow from decay of correlations, even in the most simple cases, such as when we consider the doubling map $f:[0,1]\to[0,1]$, given by $f(x)=2x \quad\mbox{mod } 1$. In this case we have decay of correlations against $L^1(\p)$, where $\p$ can be taken as Lebesgue measure and $\mathcal C_1$ is the space of functions of bounded variation. For simplicity, assume that $\{X_0>u_n\}$ corresponds to a small (for large $n$) interval around $\zeta=0$. Then the event $\{X_{j}\leq u_n\}$ is made out of $2^{j+1}$ disjoint intervals, which means that $\|\I_{ X_{j}\leq u_n}\|_{\mathcal C_1}>2^{j+2}$. This precludes any possibility of proving $D(u_n)$ because $\alpha(n,t)$ needs to be uniform, regardless of the number of $p$ random variables taken on the left block, which means that $j$ could be arbitrarily large and cannot be cancelled out by the rate of decay of correlations $\gamma$.
\end{remark}

\subsection{Some geometric structure}
\label{ssec:geom}

Suppose that the time series $X_0, X_1,\ldots$ arises from such a system simply by evaluating a given  observable $\varphi:\X\to\R\cup\{\pm\infty\}$ along the orbits of the system, or in other words, the time evolution given by successive iterations by $f$:
\begin{equation}
\label{eq:def-stat-stoch-proc-DS} X_n=\varphi\circ f^n,\quad \mbox{for
each } n\in {\mathbb N}.
\end{equation}
Clearly, $X_0, X_1,\ldots$ defined in this way is not an independent sequence.  However, $f$-invariance of $\p$ guarantees that this stochastic process is stationary.

We recall that $\zeta$ is said to be a periodic point of prime period $p$, if $f^p(\zeta)=\zeta$ and for all $j=1, \ldots, p-1$ we have $f^j(\zeta)\neq \zeta$. 

We suppose that the r.v. $\varphi:\X\to\R\cup\{\pm\infty\}$
achieves a global maximum at $\zeta\in \X$ (we allow
$\varphi(\zeta)=+\infty$). 
We assume that $\varphi$ and $\p$ are sufficiently regular so that:

\begin{enumerate}

\item[\namedlabel{item:U-ball}{(R1)}] 
for $u$ sufficiently close to $u_F:=\varphi(\zeta)$,  the event 
\begin{equation*}
\label{def:U}
U(u):=\{x\in\X:\; \varphi(x)>u\}=\{X_0>u\}
\end{equation*} is homeomorphic to an open ball centred at $\zeta$. Moreover, the quantity $\p(U(u))$, as a function of $u$, varies continuously on a neighbourhood of $u_F$.

\item[\namedlabel{item:repeller}{(R2)}]  
 Whenever we say that $\zeta\in \X$ is a periodic point, of prime period 
$p\in\N$, then we assume that the periodicity of $\zeta$ implies that for all large $u$, $\{X_0>u\}\cap f^{-p}(\{X_0>u\})\neq\emptyset$ and the fact that the prime period is $p$ implies that $\{X_0>u\}\cap f^{-j}(\{X_0>u\})=\emptyset$ for all $j=1,\ldots,p-1$. Moreover, we also assume that 
  $\zeta$ is \emph{repelling}, which from the topological side means that $\cap_{k\ge 0} f^{-kp}(U(u))=\{\zeta\}$ for all large $u$, and from a metric side means that there exists $0<\theta<1$ such that 
 $$\p\left(\{X_0>u\}\cap f^{-p}(\{X_0>u\})\right)\sim(1-\theta)\p(X_0>u),$$ for all $u$ sufficiently large. 

\end{enumerate}

\begin{remark}
\label{rem:EI-existence}
Note that under condition \ref{item:repeller} the limit in \eqref{eq:OBrien-EI} (and later, that in \eqref{eq:def-EI}) always exists. Also observe that the assumption of decay of correlations against $L^1$ precludes the existence of periodic points that are not repelling. Hence, for such systems, as long as the measure is nicely behaved (such as measures with the Gibbs property) then condition \ref{item:U-ball} holds and at periodic points we can show (see \cite[Lemma~3.1]{FFT12}) that condition \ref{item:repeller}  holds as well, which ultimately imply the existence of an EI. Also, see Section~\ref{ssec:cty} below.
\end{remark}

We are interested in studying the extremal behaviour of the stochastic process $X_0, X_1,\ldots$ which is tied with the occurrence of exceedances of high levels $u$. The occurrence of an exceedance at time $j\in\N_0$ means that the event $\{X_j>u\}$ occurs, where $u$ is close to $u_F$. Observe that a realisation of the stochastic process $X_0, X_1,\ldots$ is achieved if we pick, at random and according to the measure $\p$, a point $x\in\X$, compute its orbit and evaluate $\varphi$ along it. Then saying that an exceedance occurs at time $j$ means that the orbit of the point $x$ hits the ball $U(u)$ at time $j$, \ie $f^j(x)\in U(u)$.

\subsection{Main results}
\label{subsec:main-results}

Next we give sharper error terms for the distributional limit of the partial maximum and of the first hitting time, which have both to do with rare events corresponding to entrances in shrinking balls around a point $\zeta$. The basic assumption is decay of correlations against $L^1$ observables. 

These results relate with the main result from \cite{A04}, where sharp error terms of this type were obtained for the first hitting time of $\psi$ and $\phi$ mixing processes, arising from shift dynamics over a finite alphabet, and where the hitting targets were cylinders rather than balls.

In \cite{K12}, a similar result to that of \cite{A04} was obtained using a very powerful technique developed in \cite{KL09}, giving essentially the same estimates, in the more general context of balls around some $\zeta$ (rather than cylinders), for systems with a spectral gap for the corresponding Perron-Frobenius operator.  We  comment more on the comparisons between these works and ours in Remark~\ref{rmk:comparison to Ab and Kel}.  

As we have seen, from \cite{FFT10,FFT11}, the existence of EVLs and HTS are two sides of the same coin. In here, we present the error terms in these two different contexts. On one hand, our tools are developed on EVL context which makes it natural to write Theorem~\ref{thm:sharp-error-EVL} below. However, the normalising sequences on the EVL context are traditionally designed in such a way that they implicitly define a relation between the radii of the target balls around $\zeta$ and the time scale $\tau$. On the other hand, in the HTS context, this intrinsic relation does not exist, which leads to complications in the proof, but then allows us to apply our results to obtain estimates for the escape rate in Corollary~\ref{cor:zero hole}: hence, we also state Theorem~\ref{thm:sharp-error-HTS}.

\subsubsection{Improved error terms for EVLs}
\label{subsubsec:EVL err}

Next we present our main result of this section in the context of EVLs, which provides sharper error terms for the convergence in distribution of the partial maximum of stochastic processes as defined in \eqref{eq:def-stat-stoch-proc-DS}.

\begin{theorem}
\label{thm:sharp-error-EVL}
Assume that the system has decay of correlations of observables in a Banach space $\mathcal C$ against observables in $ L^1$ with rate function $\gamma:\N\to \R$ and there exists $\delta>0$ such that 
$n^{1+\delta}\gamma(n)\to 0$, as $n\to\infty$. Let $u_n$ be as in \eqref{eq:un}.    Assume that there exists $q\in\N_0$ such that
\begin{equation*}
\label{eq:q-def EVL}
q:=\min\left\{j\in\N_0: \lim_{n\to\infty}R(A_n^{(j)})=\infty\right\}.
\end{equation*}
For each $n\in\N$, let $A_n:=\A_n$, $R_n:=R(\A_n)$, where $R$ is defined as in \eqref{eq:first-return}, and  let $k_n, t_n$ be such that 
\begin{multline*}k_nt_n\p(A_n)+n\gamma(t_n)\left(1+\frac{n\p(A_n)}{k_n}\right)+\frac{(n\p(A_n))^2}{k_n}\\=\inf_{k, t\in\N: k t<n}\left\{kt\p(A_n)+n \gamma(t)\left(1+\frac{n\p(A_n)}{k}\right)+\frac{(n\p(A_n))^2}{k}\right\}.
\end{multline*}
Assume that $\I_{A_n}\in \mathcal C$ and there exists $M>0$ such that  $\|\I_{A_n}\|_\mathcal C\leq M$ for all $n\in\N$.

Then there exists $C>0$ such that for $n$ large, 
\begin{multline*}
\left|\p(M_n\leq u_n)-\e^{-\theta \tau}\right|\leq C\e^{-\theta\tau}\Bigg(\left|\theta\tau-n\p(A_n)\right|+ k_nt_n\frac{\theta\tau} n+n\gamma(t_n)\left(1+\frac{\theta\tau}{k_n}\right)+\frac{(\theta\tau)^2}{k_n}\\+\theta\tau\sum_{j=R_n}^{\ell_n-1}\gamma(j)\Bigg),
\end{multline*}
where $\ell_n=\lfloor n/k_n\rfloor-t_n$ and the EI $\theta$ is given by equation \eqref{eq:OBrien-EI}.
\end{theorem}

\begin{remark}
Note that condition \ref{item:U-ball} guarantees the existence of sequences $(u_n)_{n\in\N}$ satisfying \eqref{eq:un} (see \cite{FFT11} for details). As mentioned in Remark~\ref{rem:EI-existence}, condition \ref{item:repeller} guarantees the existence of the limit in \eqref{eq:OBrien-EI} that ultimately allows us to define the EI. Moreover, the fact that the observable $\varphi$ achieves a global maximum at $\zeta$ together with \ref{item:U-ball} and \ref{item:repeller}, allows us to show that $q$ appearing in Theorem~\ref{thm:sharp-error-EVL}  and Theorem~\ref{thm:sharp-error-HTS} is well defined (see \cite[Lemma~3.1]{AFV12} and \cite[Proof of Theorem 2]{FFT13}). See also Remarks~\ref{rem:application-continuous} and \ref{rem:discontinuous}.

\end{remark}

\begin{remark}
\label{rem:sequences}
Note that the existence of $\delta>0$ such that $n^{1+\delta}\gamma(n)\to 0$, as $n\to\infty$, allows us to define sequences given by $t'_n=n^{\frac{1}{1+\delta}}$ and $k'_n=n^{\frac{\delta}{2+2\delta}}$, which are such that 
$$
\frac{k'_nt'_n}{n}+n\gamma(t'_n)+\frac{1}{k'_n}\leq n^{-\frac{\delta}{2+2\delta}}+n\gamma(n^{\frac{1}{1+\delta}})+ n^{-\frac{\delta}{2+2\delta}}\xrightarrow[n\to\infty]{}0,
$$ 
since, by assumption, we have $\lim_{n\to\infty}n\gamma(n^{\frac{1}{1+\delta}})=0$.
Hence, the optimal choice of $k_n$ and $t_n$ guarantees that, under the hypothesis of Theorem~\ref{thm:sharp-error-EVL}, we must have:  
$$k_nt_n\p(A_n)+n\gamma(t_n)\left(1+\frac{n\p(A_n)}{k_n}\right)+(n\p(A_n))^2/k_n\to0 \qquad \text{ as } n\to\infty,$$
which, in particular, implies that: 
	\begin{align}
	&k_n, t_n\to \infty \quad\mbox{as $n\to\infty$};\nonumber\\
	&k_nt_n=\o(n).\label{eq:tn-kn-relation}
	\end{align}
If we define a priori rates for the above divergences/convergences along with the rate of growth of $R_n$, then the constant $C$ in the theorem depends only on those rates, on $M$ and on a bound on the deviation from the limiting constants in \eqref{eq:un} and \ref{item:repeller}.  That is to say that to apply this theorem, we don't need to use the optimal sequences $(k_n)_{n\in\N}, (t_n)_{n\in\N}$, so long as the conditions above hold.  Note that $C$ can be taken independently of $\tau$.
	
For example, 
taking $k'_n$ and $t'_n$ as above, we have
\begin{align*}
\left|\p(M_n\leq u_n)-\e^{-\theta \tau}\right|&\leq C\e^{-\theta\tau}\left(k'_nt'_n\frac{\tau} n+n\gamma(t'_n)\left(1+\frac{\tau}{k_n'}\right)+\frac{\tau^2}{k'_n}+\tau\sum_{j=R_n}^{\ell_n-1}\gamma(j)\right)\\
&\leq  C\e^{-\theta\tau}\left(n^{-\frac{\delta}{2+2\delta}}(\tau+\tau^2)+ n\gamma(n^{\frac{1}{1+\delta}})(1+\tau n^{-\frac{\delta}{2+2\delta}})+ \frac{R_n^{-(1+\delta)}\kappa\tau}{1+\delta}\right).
 \end{align*}
\end{remark}

For comments on the existence of $q$ as in the theorem and on the existence of the limit in \eqref{eq:OBrien-EI}, see Remark~\ref{rem:application-continuous} below.

\subsubsection{Improved error terms for HTS}

Here we prove a result analogous to Theorem~\ref{thm:sharp-error-EVL}.  The key differences are that it applies to HTS rather than EVLs, it is more directly applicable to balls of general diameter, and that it explicitly decouples the time scale $\tau$ from the radius of the ball $\eps$. 

In what follows $B_\eps(\zeta)$ denotes the open ball of radius $\eps$, around the point $\zeta\in\X$, w.r.t. a given metric on $\X$. Also set $A_\eps^{(0)}(\zeta):=B_\eps(\zeta)$ and, for each $q\in\N$, let
$$
\A_\eps(\zeta):=B_\eps(\zeta)\cap\bigcap_{i=1}^{q}f^{-i}((B_\eps(\zeta))^c).
$$

\begin{theorem}
\label{thm:sharp-error-HTS}
Assume that the system has decay of correlations of observables in a Banach space $\mathcal C$ against observables in $ L^1$ with rate function $\gamma:\N\to \R$ and there exists $\delta>0$ such that
$n^{1+\delta}\gamma(n)\to 0$, as $n\to\infty$. Fix some point $\zeta\in\X$ and assume that there exists $q\in\N_0$ such that 
\begin{equation*}
\label{eq:q-def HTS}
q:=\min\left\{j\in\N_0: \lim_{\eps\to0}R(A_\eps^{(j)}(\zeta))=\infty\right\}.
\end{equation*}
 For each $\eps>0$, let $B_\eps:=B_\eps(\zeta)$, $A_\eps:=\A_\eps(\zeta)$, $R_\eps:=R(\A_\eps(\zeta))$, where $R$ is defined as in \eqref{eq:first-return}, and $k_\eps, t_\eps\in\N$ be such that 
$$k_\eps t_\eps\p(B_\eps)+\frac{\gamma(t_\eps)}{\p(B_\eps)}+\frac{1}{k_\eps}=\inf_{k, t\in\N: kt<\p(B_\eps)^{-1}}\left\{k t\p(B_\eps)+\frac{\gamma(t)}{\p(B_\eps)}+\frac{1}{k}\right\}.$$
Let  $\ell_\eps=\lfloor\lfloor\p(B_\eps)^{-1}\rfloor/k_\eps\rfloor-t_\eps$ and $L_\eps:=1-\ell_\eps\p(A_\eps)$. Assume that $\I_{A_\eps}\in\mathcal C$ and there exists $M>0$ such that  $\|\I_{A_\eps}\|_\mathcal C\leq M$ for all $\eps>0$. 

Then there exists $C>0$, depending on $\eps$ but not on $\tau$, such that for all $\eps>0$ small enough and all $\tau>0$ 
\begin{align*}
\Big|\p\left(r_{B_\eps(\zeta)}>\frac \tau{\p(B_\eps)}\right)-\e^{-\theta \tau}\Big|\leq C \left(\tau^2\alpha_\eps\Gamma_\eps+\frac {\tau^2}{k_\eps}\Gamma_\eps+\frac {\tau^3}{k_\eps}\alpha_\eps\Gamma_\eps\right) \e^{ -(\theta- k_\eps\Upsilon_{A_\eps}L_\eps^{-1})\tau},
\end{align*}
where $\Gamma_\eps=\left(k_\eps t_\eps\p(A_\eps)+(\p(B_\eps))^{-1}\gamma(t_\eps)+ k_\eps^{-1}+\sum_{j=R_\eps}^{\ell_\eps-1}\gamma(j)\right)$, $\alpha_\eps=|\theta-\frac{\p(A_\eps)}{\p(B_\eps)}+t_\eps k_\eps\p(A_\eps)|$, $\Upsilon_{A_\eps}$ is given as in Lemma~\ref{lem:main-thing} 
and
 \begin{equation}
\label{eq:def-EI}
\theta=\lim_{\eps\to0}\frac{\p(A_\eps)}{\p(B_\eps)}.
\end{equation}
\end{theorem}

\begin{remark}
\label{rem:application-continuous}
If $f$ is continuous, for example, and $\p$ is sufficiently regular so that the content of condition \ref{item:U-ball} holds for balls around all $\zeta\in\X$, then if $\zeta$ is not periodic then Theorem~\ref{thm:sharp-error-HTS} holds with $q=0$ and $\theta=1$. Moreover, if $\zeta$ is a repelling periodic point of prime period $p$, in the sense described in condition \ref{item:repeller}, then Theorem~\ref{thm:sharp-error-HTS} also applies with $q=p$ and $\theta$ is given as in \ref{item:repeller} by the backward contraction rate at $\zeta$. Recall that as discussed in Remark~\ref{rem:EI-existence} for systems with decay of correlations against $L^1$ and as long as the invariant measure is sufficiently regular (having the Gibbs property suffices) then the limits in the definition of the EI always exist.
\end{remark}

\begin{remark}
\label{rem:discontinuous}
Theorem~\ref{thm:sharp-error-HTS} can also be applied to discontinuity points $\zeta$ of $f$ as considered in \cite[Section~3.3]{AFV12} so that, if for example, there exist $p^-$ and $p^+$ such that $f^{p^+}(\zeta^+)=\zeta=f^{p^-}(\zeta^-)$ (see \cite[Section~3.3]{AFV12} for details) then $q=\max\{p^-,p^+\}$ and $\theta$ is given by the formulas in \cite[Proposition~3.4]{AFV12}.    
\end{remark}

\begin{remark}
Note that the existence of $\delta>0$ such that $n^{1+\delta}\gamma(n)\to 0$, as $n\to\infty$, allows us to find `sequences' $(t'_\eps)_\eps$ and $(k'_\eps)_\eps$, just as in Remark~\ref{rem:sequences},  such that  $\lim_{\eps\to0}k'_\eps t'_\eps\p(A_\eps)+\gamma(t'_\eps)/\p(B_\eps)+1/k'_\eps=0$. It follows that for the optimal `sequences' $(t_\eps)_\eps$ and $(k_\eps)_\eps$, we must have
$$k_\eps t_\eps\p(A_\eps)+\gamma(t_\eps)/\p(B_\eps)+1/k_\eps\to0, \qquad \text{ as } \eps\to0,$$
which, in particular, implies that: 
	\begin{align}
	&k_\eps, t_\eps\to \infty \quad\mbox{as $\eps\to0$};\nonumber\\
	&k_\eps t_\eps=\o(\p(B_\eps)^{-1})\label{eq:t_eps-k_eps-relation}.
	\end{align}

So, as in Remark~\ref{rem:sequences}, we can fix growth rates on the parameters in this theorem to show that the constant $C$ depends only on those rates.
\label{rmk:decay consts}
\end{remark}

\begin{remark}
\label{rem:error-terms}
Observe that the error terms are dominated by $\Gamma_\eps$ and to some extent by $\alpha_\eps$. There are four terms in $\Gamma_\eps$. The first and the third terms of $\Gamma_\eps$ result from the fact that we use a blocking argument and essentially take into account the balance between the size of $\p(A_\eps)$, the number of blocks and the size of the gaps between the blocks. As seen in Remark~\ref{rem:sequences}, this ultimately leads to an error that typically goes down like  $P(A_\epsilon)^\delta$ for some $0 < \delta < 1$. For the second term we need to balance the size of gaps between the blocks and the loss of memory of the system. Since, as observed above, decay of correlations against $L^1$ is a strong mixing property, then typically $\gamma$ decays exponentially fast which means that this term is usually negligible when compared to the others. The fourth term in $\Gamma_\eps$ accounts for the affect of fast recurrence from $A_\eps$ to $A_\eps$ itself. So, if $\gamma$ decays exponentially fast, then the fourth term should decay as $\gamma(R_\eps)$. The quantity $R_\eps$ has been studied in \cite{ACS00,AL13,AV08,HV10,STV02}, for example, and we generally expect it to behave like $-\log(\p(A_\eps))$, which means that the fourth term should also decay like a power of $\p(A_\eps)$. As a consequence the behaviour of $\Gamma_\eps$ should be ruled by a trade-off between the first, the third and the fourth terms of $\Gamma_\eps$.

Regarding $\alpha_\eps$, note that for a sufficiently regular dynamical system $f$ and invariant measure $\p$, the dominant term should be again the first term of $\Gamma_\eps$.
\end{remark}

\begin{remark}
We note that, when compared with the sharp estimates in \cite{A04,K12} in terms of $\tau$ we have a loss here from $\tau\e^{-\xi_{A_\eps}\tau}$ in \cite{A04} (or $\tau\e^{-\xi_\eps\tau}$ in \cite{K12}), to $\tau^3\e^{-(\theta-k_\eps\Upsilon_\eps)\tau}$, which is explained by the fact that we compute the error terms with respect to the asymptotic limit $\e^{-\theta\tau}$ with no correcting factors such as $\xi_{A_\eps}$, used in \cite{A04}, and $\xi_\eps$, in \cite{K12}. Observe that even in the ideal i.i.d.\ case, as noted in Remark~\ref{rem:rates-general-case}, when the error terms are computed with respect to the asymptotic limit ($\e^{-\tau}$, in this case) then an error term $\tau^2\e^{-\tau}$ already appears (see \cite[Theorem~2.4.2]{LLR83}). Moreover, the deeper analysis we perform here explains how  $\theta-k_\eps\Upsilon_\eps$ goes to $0$, as $\eps\to\infty$, which, as we have seen, depends on the fast recurrence of the point $\zeta$ to itself.
\label{rmk:comparison to Ab and Kel}
\end{remark}

\begin{remark}
In both Theorems~\ref{thm:sharp-error-EVL} and \ref{thm:sharp-error-HTS}, we assume, for simplicity, that $\I_{A_n}$ and $\I_{A_\eps}$ belong to $\mathcal C$. This is true in applications, such as those mentioned in Remark~\ref{rem:decay-L1}, in which the Banach space $\mathcal C$ is the space of functions of bounded variation or bounded quasi-H\"older norm considered in \cite{S00}. However, the statement of the theorems can still be proved even in cases when these assumptions do not hold. This is the case when one takes for $\mathcal C$ the space of H\"older continuous functions, for example. Nevertheless, using a suitable H\"older continuous approximation we can still recover the result. See discussion in \cite[Section~5.1]{F13} and in particular Proposition~5.2.

\end{remark}

\subsection{Escape rates in the zero-hole limit}
\label{ssec:zero-hole}

One way of studying the recurrence properties of a system is to fix a hole $B_\eps(\zeta)$ around some chosen point $\zeta$ and compute the rate of escape of mass through the hole, i.e., find the limit, if it exists,
$$-\lim_{t\to \infty}\frac1t\log\p\left(r_{B_\eps(\zeta)}>t\right).$$

Moreover, one can consider, as in \cite{KL09, FP12}, what happens when the size of the ball goes to zero too.  In those papers, it is shown that we should expect this to scale with $\p(B_\eps(\zeta))$.  Theorem~\ref{thm:sharp-error-HTS} gives error estimates independently of the time scale $\tau$, which yields the following corollary.

\begin{corollary}
Suppose that the system is as in Theorem~\ref{thm:sharp-error-HTS}.  Then
$$-\lim_{\eps\to0}\frac1{\p(B_\eps(\zeta))}\limsup_{t}\frac1t\log\p\left(r_{B_\eps(\zeta)}>t\right)\ge  \theta.$$
\label{cor:zero hole}
\end{corollary}

Note that for many dynamical systems it is known that the upper limit is also $\theta$, see for example the systems in the papers referenced above, but we were unable to prove this in the generality of the setting above.

\begin{proof}
In Theorem~\ref{thm:sharp-error-HTS}, the error estimate is  
\begin{align*}
\Big|\p\left(r_{B_\eps(\zeta)}>\frac \tau{\p(B_\eps)}\right)-\e^{-\theta \tau}\Big|\leq C \left(\tau^2\alpha_\eps\Gamma_\eps+\frac {\tau^2}{k_\eps}\Gamma_\eps+\frac {\tau^3}{k_\eps}\alpha_\eps\Gamma_\eps\right) \e^{ -(\theta- k_\eps\Upsilon_{A_\eps})\tau},
\end{align*}
where $\Gamma_\eps=\left(k_\eps t_\eps\p(A_\eps)+(\p(B_\eps))^{-1}\gamma(t_\eps)+ \frac1{k_\eps}+\sum_{j=R_\eps}^{\ell_\eps-1}\gamma(j)\right)$, $\alpha_\eps=|\theta-\frac{\p(A_\eps)}{\p(B_\eps)}+t_\eps k_\eps\p(A_\eps)|$ and
 \begin{equation*}
\theta=\lim_{\eps\to0}\frac{\p(A_\eps)}{\p(B_\eps(\zeta))}.
\end{equation*}

Recall that $n^{1+\delta}\gamma(n)\to 0$ as $n\to \infty$. 
Then assuming that $\delta\in (0,1)$ and setting $k_\eps=\frac1{\p(B_\eps)^{\frac{\delta}{2+2\delta}}}$ and $t_\eps=\frac1{\p(B_\eps)^{\frac{1}{1+\delta}}}$,

we deduce that
$$\Gamma_\eps\le \p(B_\eps)^{\frac{\delta}{2+2\delta}}\left(\frac{\p(A_\eps)}{\p(B_\eps)}+1\right)+\p(B_\eps)^{-1}\gamma\left(\p(B_\eps)^{-\frac{1}{1+\delta}}\right) + C R_\eps^{1+\delta}$$
for $C$ depending only on $\delta$
and 
$$\alpha_\eps\le \left|\theta-\frac{\p(A_\eps)}{\p(B_\eps)}\right|+\p(B_\eps)^{\frac{\delta}{2+2\delta}}\frac{\p(A_\eps)}{\p(B_\eps)}.$$

In particular,
$$\limsup_{\tau}\frac1\tau\log\p\left(r_{B_\eps(\zeta)}>\tau\right)\le -(\theta- k_\eps\Upsilon_{A_\eps}) \p(B_\eps(\zeta)).$$
Therefore
$$-\lim_{\eps\to0}\frac1{\p(B_\eps(\zeta))}\limsup_{\tau}\frac1\tau\log\p\left(r_{B_\eps(\zeta)}>\tau\right)\ge  \theta,$$
as required.
\end{proof}

\section{Preparatory lemmas}
\label{sec:Lemmata}

In this section, we give the estimates necessary to prove Theorems~\ref{thm:sharp-error-EVL} and \ref{thm:sharp-error-HTS}. As in previous such arguments, the core of the strategy is a blocking argument as described just before Lemma~\ref{lem:time-gap-1}. The estimates here are enhanced versions of Lemmas~\ref{lem:time-gap-1}, \ref{lem:inductive-step-1} and Proposition~\ref{prop:main-estimate-1}, obtained using the stronger assumption given by the decay of correlations against $L^1$ observables.

Next we give an estimate on the error resulting from neglecting the random variables in the gaps between the blocks.

\begin{lemma}
\label{lem:time-gap}Assume that the systems has decay of correlations against $L^1$ with a non-increasing rate function $\gamma:\N\to \R$ 
such that $\gamma(n)\to 0$, as $n\to\infty$.
For any fixed $A\in \mathcal B$ and positive integers $s,t,t',m$, with $t< m$, we have:
\begin{equation*}
\left|\p(\mathscr W_{0,s+t'+m}(A))-\p(\mathscr W_{0,s}(A)\cap \mathscr W_{s+t',m}(A))\right|\leq t'\left[\p(A)+\|\I_A\|_{\mathcal C}\,\gamma(t)\right]\p(\mathscr W_{0,m-t}(A)).
\end{equation*}
 \end{lemma}
 \begin{proof}
By stationarity, we have for $\mathscr W_{i,n}=\mathscr W_{i,n}(A)$,
\begin{multline*}
\p(\mathscr W_{0,s}\cap \mathscr W_{s+t',m})-\p(\mathscr W_{0,s+t'+m})=\p(\mathscr W_{0,s}\cap \mathscr W_{s,t'}^c\cap\mathscr W_{s+t',m})\leq
\p(\mathscr W_{0,t'}^c\cap\mathscr W_{t',m})\\
\leq \p(\mathscr W_{0,t'}^c\cap\mathscr W_{t'+t,m-t}) \leq \sum_{j=0}^{t'-1}\p(f^{-j}(A)\cap \mathscr W_{t'+t,m-t}) =  \sum_{j=0}^{t'-1}\p(A\cap \mathscr W_{t'+t-j,m-t})
\end{multline*}
Using stationarity, decay of correlations against $L^1$, with $\phi=\I_A$ and $\psi=\I_{\mathscr W_{0,m-t}}$, and the fact that $\gamma$ in non-increasing, it follows that
\begin{align*}
 \sum_{j=0}^{t'-1}\p\left(A\cap \mathscr W_{t'+t-j,m-t}\right)&\leq \sum_{j=0}^{t'-1} \left(\p(A)\p(\mathscr W_{t'+t-j,m-t})+ \|\I_A\|_{\mathcal C} \p(\mathscr W_{t'+t-j,m-t}) \gamma(t'+t-j)\right)\\ 
 &\leq \sum_{j=0}^{t'-1}\left[\p(A)+ \|\I_A\|_{\mathcal C}\,\gamma(t)\right]\p(\mathscr W_{0,m-t}),
\end{align*}

and the result follows immediately.
\end{proof}
 
In the result that follows we give an estimate for the probability of not visiting the set $A$ within a block, which is the cornerstone of the recursive arguments mounted in Lemmas~\ref{lem:crucial-estimate} and \ref{lem:main-thing}.

\begin{lemma}
\label{lem:inductive-step}Assume that the system has decay of correlations against $L^1$ with a non-increasing rate function $\gamma:\N\to \R$. 
such that $\gamma(n)\to 0$, as $n\to\infty$.
Then for any fixed $A,B\in \mathcal B$ and positive integers $s,t,m$, we have:
\begin{equation*}
\Big|\p(\mathscr W_{0,s}(A)\cap f^{-(s+t)}(B))-\p(B)(1-s\p(A))\Big|\leq \Xi_{A,s} \;\p(B),
\end{equation*}
where
\begin{align*}
\Xi_{A,s}:=& \|\I_A\|_\mathcal Cs \gamma (t)
+\|\I_A\|_{\mathcal C}(s-R(A))\left[\p(A)+\|\I_A\|_{\mathcal C}\gamma(t)\right]\sum_{q=R(A)}^{s-1}\gamma(q)\\& 
+s(s-R(A))\left[(\p(A))^2+\p(A)\|\I_A\|_{\mathcal C}\gamma(t)\right].
\end{align*}

 \end{lemma}
\begin{proof}
Observe that for $\mathscr W_{i,n}=\mathscr W_{i,n}(A)$,
\begin{multline}
\label{eq:triangular1}
\big|\p(\mathscr W_{0,s}\cap f^{-(s+t)}(B) )-\p(B)(1-s\p(A))\big|\leq \left|s\p(A)\p(B)-\sum_{j=0}^{s-1}\p(A\cap f^{-(s+t-j)}(B))\right| + \\
+\left|\p(\mathscr W_{0,s}\cap f^{-(s+t)}(B) )-\p(B)+\sum_{j=0}^{s-1}\p(A\cap f^{-(s+t-j)}(B))\right|.
\end{multline}
Using stationarity, decay of correlations against $L^1$ and the fact that $\gamma$ is non-increasing, it follows immediately that for the first term on right:
\begin{align}
\label{eq:term1}
\left|s\p(A)\p(B)-\sum_{j=0}^{s-1}\p(A\cap f^{-(s+t-j)}(B))\right|&\leq s \gamma (t)\|\I_A\|_\mathcal C\|\I_{B}\|_{1}= s \gamma (t)\|\I_A\|_\mathcal C\p(B).
\end{align}
For the second term on the right we have
\[
\p(\mathscr W_{0,s}\cap f^{-(s+t)}(B))=\p(f^{-(s+t)}(B))-\p(\mathscr W_{0,s}^c\cap f^{-(s+t)}(B))=\p(B)-\p(\mathscr W_{0,s}^c\cap f^{-(s+t)}(B)).
\]
Moreover since $\mathscr W_{0,s}^c\cap f^{-(s+t)}(B)=\cup_{i=0}^{s-1}f^{-i}(A)\cap f^{-(s+t)}(B)$, by the formula for the probability of the combination of events (see the beginning of Section~4 of \cite{F50}, for example), it follows that
\[
0\leq \sum_{j=0}^{s-1}\p(f^{-j}(A)\cap f^{-(s+t)}(B))-\p(\mathscr W_{0,s}^c\cap f^{-(s+t)}(B)) \leq \sum_{j=0}^{s-1}\sum_{i>j}^{s-1}\p(f^{-j}(A)\cap f^{-i}(A)\cap f^{-(s+t)}(B)).
\]
Hence, by definition of $R(A)$ (note that $A\cap f^{-i}(A)\neq \emptyset$ implies that $i\geq R(A)$), stationarity and decay of correlations against $L^1$ (with $\phi=\I_A$ and $\psi=\I_{A\cap f^{-(s+t-i)}(B)}$) we have 
\begin{align*}
\Big|\p(\mathscr W_{0,s}&\cap f^{-(s+t)}(B))-\p(B)+\sum_{j=0}^{s-1}\p(A\cap f^{-(s+t-j)}(B))\Big|\\
&\leq \sum_{j=0}^{s-1}\sum_{i= j+R(A)}^{s-1}\p\left(A\cap f^{-(i-j)}(A)\cap f^{-(s+t-j)}(B)\right)\nonumber\\
&\leq\sum_{j=0}^{s-1}\sum_{i= j+R(A)}^{s-1} \p(A)\p\left(A\cap f^{-(s+t-i)}(B)\right)+\sum_{j=0}^{s-1}\sum_{i= j+R(A)}^{s-1} \|\I_A\|_\mathcal C\gamma(i-j)\p\left(A\cap f^{-(s+t-i)}(B)\right)\\&=: I+I\!\!I.
\end{align*}
Using stationarity, decay of correlations against $L^1$ (with $\phi=\I_A$ and $\psi=\I_{ B}$), the fact that $s+t-i>t$ and $\gamma$ is non-increasing, we have
\begin{equation*}
I\leq \sum_{j=0}^{s-1}\sum_{i= j+R(A)}^{s-1}\p(A)\left[\p(A)+\|\I_A\|_\mathcal C\gamma(t)\right]\p(B)
\leq \p(B)s(s-R(A))\left[(\p(A))^2+\p(A)\|\I_A\|_{\mathcal C}\gamma(t)\right].
\end{equation*}
and 
\begin{align*}
I\!\!I&\leq \sum_{j=0}^{s-1}\sum_{i= j+R(A)}^{s-1} \|\I_A\|_\mathcal C\gamma(i-j)\left[\p(A)+\|\I_A\|_\mathcal C\gamma(t)\right]\p(B)\\&\leq  \p(B)\|\I_A\|_\mathcal C \left[\p(A)+\|\I_A\|_\mathcal C\gamma(t)\right]\sum_{q=R(A)}^{s-1}(s-q)\gamma(q)\\&\leq  \p(B)\|\I_A\|_\mathcal C (s-R(A)) \left[\p(A)+\|\I_A\|_\mathcal C\gamma(t)\right]\sum_{q=R(A)}^{s-1}\gamma(q).
\end{align*}

The result now follows from plugging \eqref{eq:term1} and the estimates on $I$ and $I\!\!I$ into \eqref{eq:triangular1}.
\end{proof}

\begin{lemma}
\label{lem:crucial-estimate}
Assume that the systems has decay of correlations against $L^1$ with a non-increasing rate function $\gamma:\N\to \R$ 
such that $\gamma(n)\to 0$, as $n\to\infty$.
Let $A,D\in \mathcal B$ and $\ell,t$ be positive integers, such that $\ell\p(A)<1$. Also let $L=1-\ell\p(A)$. Then, for every $i\in\N$, we have
\begin{equation*}
\p\left(\bigcap_{j=0}^{i-1}\mathscr W_{j(\ell+t), \ell}(A)\cap f^{-i(\ell+t)}(D)\right)\leq (L+\Xi_{A,\ell})^i\;\p(D).
\end{equation*}
\end{lemma}

\begin{proof}
We begin by defining $E_0=D$ and $E_j=\bigcap_{s=0}^{j-1}\mathscr W_{s(\ell+t), \ell}(A)\cap f^{-j(\ell+t)}(D)$, for every $j=1,\ldots, i$. Now, we note that
$$
\p(E_i)-L^i\p(D)=\sum_{j=0}^{i-1} L^j\left(\p(E_{i-j})-L\p(E_{i-j-1})\right).
$$
Observing that, for every $j$, we have $E_j=\mathscr W_{0,\ell}(A) \cap f^{-(\ell+t)}( E_{j-1})$, we can use Lemma~\ref{lem:inductive-step}, with $s=\ell$ and $B=E_{i-j-1}$, to obtain that, for every $j=1,\ldots, i$, we have 
$$
\left|\p(E_{i-j})-L\p(E_{i-j-1})\right|\leq \Xi_{A,\ell} \p(E_{i-j-1}).
$$
Hence, 
\begin{equation}
\label{eq:recursive-formula}
|\p(E_i)-L^i\p(D)|\leq \sum_{j=0}^{i-1} L^j \Xi_{A,\ell} \p(E_{i-j-1}).
\end{equation}

We want to show that for all $i\in\N$, we have 
\begin{equation}
\label{eq:induction-1}
\p(E_i)\leq (L+\Xi_{A,\ell})^i\p(D),
\end{equation} which we will show by induction. For $i=1$, the inequality holds, trivially, by applying Lemma~\ref{lem:inductive-step} with $s=\ell$ and $B=D$, to obtain that $\p(E_1)\leq (L+\Xi_{A,\ell})\p(D)$. Now, assume that inequality \eqref{eq:induction-1} holds for all $j=1,\ldots,i-1$. Then, by \eqref{eq:recursive-formula}, we have
\begin{align*}
\p(E_i)&\leq L^i\p(D)+\sum_{j=0}^{i-1} L^j \Xi_{A,\ell} \p(E_{i-j-1})=L^i\p(D)+\sum_{j=0}^{i-1} L^{i-1-j} \Xi_{A,\ell} \p(E_{j})\\
&\leq L^i\p(D)+\sum_{j=0}^{i-1} L^{i-1-j} \Xi_{A,\ell} (L+\Xi_{A,\ell})^j\p(D)\qquad\mbox{by inductive assumption}\\
&=(L+\Xi_{A,\ell})^i\;\p(D).
\end{align*}
In the last equality, we used the fact: $(a+b)^n=a^n+\sum_{j=0}^{n-1}b\,a^{n-1-j}(a+b)^j$, for all $a,b\in\R$ and $n\in\N$, which can be easily verified by induction.
\end{proof}

\begin{lemma}
\label{lem:main-thing}
Assume that the system has decay of correlations of observables $\phi$, in a Banach space $\mathcal C$, against $\psi\in L^1$, with a non-increasing rate function $\gamma:\N\to \R$, 
such that $\gamma(n)\to 0$, as $n\to\infty$.
Let $A\in \mathcal B$ and $\ell,t$ be positive integers, such that $\ell\p(A)<1$. Also let $L=1-\ell\p(A)$. Then, for every $k\in\N$, we have
\begin{equation*}
\left|\p\left(\mathscr W_{0, k(\ell+t)}(A)\right)-L^k\right|\leq k\Upsilon_A(L+\Upsilon_A)^{k-1},
\end{equation*}
where $\Upsilon_A:=t[\p(A)+\|\I_A\|_{\mathcal C}\gamma(t)]+\Xi_{A,\ell}$ and $\Xi_{A,\ell}$ is as in Lemma~\ref{lem:inductive-step}, with $s=\ell$.
\end{lemma}

\begin{proof}
For  $\mathscr W_{i,n}=\mathscr W_{i,n}(A)$ for $n\in \N$, and assuming that $\mathscr W_{0,0}=\X$, we may write
\begin{equation}
\label{eq:computation1}
\left|\p\left(\mathscr W_{0, k(\ell+t)}\right)-L^k\right|\leq \sum_{i=0}^{k-1} L^{k-1-i}\left|\p(\mathscr W_{0,(i+1)(\ell+t)})-L\p(\mathscr W_{0,i(\ell+t)})\right|. 
\end{equation}
Now, we estimate $\left|\p(\mathscr W_{0,(i+1)(\ell+t)})-L\p(\mathscr W_{0,i(\ell+t)})\right|$, which we break into two pieces.
By Lemma~\ref{lem:time-gap}, with $s=\ell, t'=t$ and $m=i(\ell+t)$, 
we have
\begin{align*}
\big|\p(\mathscr W_{0,(i+1)(\ell+t)})&-\p(\mathscr W_{0,\ell}\cap \mathscr W_{\ell+t,i(\ell+t)})\big|\leq t[\p(A)+\|\I_A\|_{\mathcal C}\gamma(t)]\p(\mathscr W_{0,i(\ell+t)-t})\\
 &\leq t[\p(A)+\|\I_A\|_{\mathcal C}\gamma(t)] \p\left(\bigcap_{j=0}^{i-1}\mathscr W_{j(\ell+t), \ell}\right).
 \end{align*}
By Lemma~\ref{lem:inductive-step}, with $s=\ell$ and $B=\mathscr W_{0,i(\ell+t)}$, we have
\begin{equation*}
\big|\p(\mathscr W_{0,\ell}\cap \mathscr W_{\ell+t,i(\ell+t)})-L\p(\mathscr W_{0,i(\ell+t)})\big|
\leq \Xi_{A,\ell} \p(\mathscr W_{0,i(\ell+t)})\leq \Xi_{A,\ell} \p\left(\bigcap_{j=0}^{i-1}\mathscr W_{j(\ell+t), \ell}\right).\end{equation*}
Using these last two estimates, recalling that $\Upsilon_A=t[\p(A)+\|\I_A\|_{\mathcal C}\gamma(t)]+\Xi_{A,\ell}$ and using Lemma~\ref{lem:crucial-estimate}, with $D=\X$, we obtain
\begin{equation}
\label{eq:important-relation}
\left|\p(\mathscr W_{0,(i+1)(\ell+t)})-L\p(\mathscr W_{0,i(\ell+t)})\right|\leq \Upsilon_A\p\left(\bigcap_{j=0}^{i-1}\mathscr W_{j(\ell+t), \ell}\right)\leq \Upsilon_A (L+\Xi_{A,\ell})^i.
\end{equation}
Plugging \eqref{eq:important-relation} into \eqref{eq:computation1} and  we obtain
\begin{align*}
\left|\p\left(\mathscr W_{0, k(\ell+t)}\right)-L^k\right|&\leq \sum_{i=0}^{k-1} L^{k-1-i}\Upsilon_A (L+\Xi_{A,\ell})^i\leq  \sum_{i=0}^{k-1} L^{k-1-i}\Upsilon_A (L+\Upsilon_A)^i\\
&=\sum_{i=0}^{k-1}\sum_{j=0}^i \binom i j L^{k-1-j}\Upsilon_A^{j+1} =
\sum_{j=0}^{k-1} L^{k-1-j}\Upsilon_A^{j+1}\sum_{i=j}^{k-1}  \binom i j \nonumber\\&= \sum_{j=0}^{k-1} L^{k-1-j}\Upsilon_A^{j+1}\binom k{j+1}= \sum_{q=1}^{k} L^{k-q}\Upsilon_A^{q}\binom k q\\
&= \left(L+\Upsilon_A\right)^k-L^k\leq k\Upsilon_A (L+\Upsilon_A)^{k-1},
\end{align*}
where the third equality derives from Fermat's combinatorial identity and the last inequality follows by the Mean Value Theorem.
\end{proof}

\begin{corollary}
\label{cor:n-estimate}
Assume that the system has decay of correlations of observables $\phi$, in a Banach space $\mathcal C$, against $\psi\in L^1$, with a non-increasing rate function $\gamma:\N\to \R$, such that $\gamma(n)\to 0$, as $n\to\infty$.
Fix $A\in \mathcal B$ and $n\in\N$. Let $\ell,k,b\in\N$ be such that $n=k\lfloor n/k \rfloor+b$ and $\lfloor n/k \rfloor\p(A)<1$. Also, consider an integer $t$ such that $t<\lfloor n/k \rfloor$, and set $\ell=\lfloor n/k \rfloor-t$, $L:=1-\ell\p(A)$.  Then
\begin{equation*}
\left|\p\left(\mathscr W_{0, n}(A)\right)-L^k\right|\leq k\Upsilon_A(L+\Upsilon_A)^{k-1}(1+L+\Upsilon_A),
\end{equation*}
where $\Upsilon_A:=t[\p(A)+\|\I_A\|_{\mathcal C}\gamma(t)]+\Xi_{A,\ell}$ and $\Xi_{A,\ell}$ is as in Lemma~\ref{lem:inductive-step}, with $s=\ell$.
\end{corollary}
\begin{proof}
Note that  for $\mathscr W_{i,n}=\mathscr W_{i,n}(A)$,
\begin{equation*}
\left|\p\left(\mathscr W_{0, n}\right)-L^k\right|\leq \left|\p\left(\mathscr W_{0, n}\right)-\p\left(\mathscr W_{0, k(\ell+t)}\right)\right|+\left|\p\left(\mathscr W_{0, k(\ell+t)}\right)-L^k\right|.
\end{equation*}
For the second term on the right we have $\left|\p\left(\mathscr W_{0, k(\ell+t)}(A)\right)-L^k\right|\leq k\Upsilon_A(L+\Upsilon_A)^{k-1}$, by Lemma~\ref{lem:main-thing}.
For the first term, by Lemma~\ref{lem:time-gap}, with $s=0$ and $t'=b$, the fact that $b\leq k$ because it is the remainder of the division of $n$ by $k$, and Lemma~\ref{lem:crucial-estimate}, with $i=k$ and $D=\X$, we have 
\begin{align*}
\left|\p\left(\mathscr W_{0, n}\right)-\p\left(\mathscr W_{0, k(\ell+t)}\right)\right|&\leq b[\p(A)+\|\I_A\|_{\mathcal C} \gamma(t)]\p(\mathscr W_{0,k(\ell+t)-t})\leq  k\Upsilon_A \p(\mathscr W_{0,k(\ell+t)-t}) \\
&\leq k\Upsilon_A\p\left(\bigcap_{j=0}^{k-1}\mathscr W_{j(\ell+t), \ell}\right)
\leq k\Upsilon_A (L+\Upsilon_A)^k.
\end{align*}
The result now follows at once.
\end{proof}

 \begin{corollaryP}
 \label{cor:extra-assumption}
 Under the same assumptions of Corollary~\ref{cor:n-estimate}, if $k\Upsilon_A<L/2$ then $$\left|\p(\mathscr W_{0,n}(A))-L^{k}\right|\leq  5k\Upsilon_A L^{k-1}.$$
\end{corollaryP}

\begin{proof}
Under the assumption that $k\Upsilon_A<L/2$ and since $L<1$, we have
\begin{align*}
\left|\p(\mathscr W_{0,k(\ell+t)})-L^k\right|&\leq k\Upsilon_A (L+\Upsilon_A)^{k-1}(1+L+\Upsilon_A)\\&\leq  k\Upsilon_A L^{k-1}\left(1+\frac1{2k}\right)^{k-1}(1+L+\Upsilon_A) \leq \e^{\frac12}\frac52k\Upsilon_A L^{k-1}.
\end{align*}
\end{proof}

The following results will be needed for the sharper estimates in the HTS setting.

\begin{corollaryP}
\label{cor:tau-included}
Under the same assumptions of Corollary~\ref{cor:n-estimate}, let $\tau>0$, note that $\tau n=\tau k (\ell +t)+\tau b$ and set $\beta=\tau k-\lfloor \tau k\rfloor$. Then, when $\lfloor \tau k\rfloor>0$ we may write:
$$
\Big|\p(\mathscr W_{0,\tau n})-L^{\lfloor \tau k\rfloor}\Big|\leq(3+\Upsilon_A)\lceil\tau k\rceil\Upsilon_A(L+\Upsilon_A)^{\lfloor \tau k\rfloor-1}.  
$$
In the case $\lfloor \tau k\rfloor=0$, we have
$$
\Big|\p(\mathscr W_{0,\tau n})-(1-\lfloor\tau k \ell\rfloor\p(A))\Big|\leq \Upsilon_A.
$$
\end{corollaryP}
\begin{proof} We start first with case $\lfloor \tau k\rfloor>0$.
The idea is to break the time interval $[0,\tau n)$ into $\lfloor \tau k\rfloor$ blocks of size $\ell+t$ plus one block of size $\beta(\ell+t)$ plus one last block of size $\tau b$.

Using Lemma~\ref{lem:time-gap}, with $s=0$, $t'=\lfloor\tau n\rfloor- \lfloor\tau k(\ell+t)\rfloor\leq \lceil\tau b\rceil$, the fact that $b< k$ (recall that $b$ is the remainder of the division of $n$ by $k$) and Lemma~\ref{lem:crucial-estimate}, with $i=\lfloor\tau k\rfloor$ and $D=\X$, it follows  for $\mathscr W_{i,n}=\mathscr W_{i,n}(A)$:
\begin{align}
\Big|\p(\mathscr W_{0,\tau n})-\p(\mathscr W_{0,\tau k(\ell+t)})\Big|&\leq \lceil \tau b\rceil[\p(A)+\|\I_A\|_{\mathcal C}\gamma(t)]\p(\mathscr W_{0,\lfloor\tau k(\ell+t)\rfloor-t})\nonumber\\
&\leq \lceil\tau k\rceil \Upsilon_A \p\left(\bigcap_{j=0}^{\lfloor\tau k\rfloor -1}\mathscr W_{j(\ell+t), \ell}\right)\leq \lceil\tau k\rceil \Upsilon_A (L+\Upsilon_A)^{\lfloor \tau k\rfloor}. 
\label{eq:approx-2}
\end{align}
If $\beta(\ell+t)\leq t$ then arguing as in \eqref{eq:approx-2}, we obtain:
 $$
 \Big|\p(\mathscr W_{0,\tau k(\ell+t)})-\p(\mathscr W_{0,\lfloor \tau k\rfloor(\ell+t)})\Big|\leq t[\p(A)+\|\I_A\|_{\mathcal C}\gamma(t)]\p(\mathscr W_{0,\lfloor \tau k\rfloor(\ell+t)-t})\leq \Upsilon_A (L+\Upsilon_A)^{\lfloor \tau k\rfloor}.
 $$
Using Lemma~\ref{lem:main-thing} (with $k$ replaced by $\lfloor\tau k\rfloor$) and the estimates just above we obtain
\begin{align*}
\Big|\p(\mathscr W_{0,\tau n})-L^{\lfloor \tau k\rfloor}\Big|\leq (3+\Upsilon_A)\lceil\tau k\rceil\Upsilon_A(L+\Upsilon_A)^{\lfloor \tau k\rfloor-1}.
\end{align*}
Now assume that $\beta(\ell+t)> t$. We have
\begin{align*}
\Big|\p(\mathscr W_{0,\tau k(\ell+t)})-L^{\lfloor \tau k\rfloor}&(1-\lfloor\beta\ell\rfloor\p(A))\Big|\leq L^{\lfloor \tau k\rfloor}|\p(\mathscr W_{0,\beta(\ell+t)}-(1-\lfloor\beta\ell\rfloor\p(A))|\\&+\sum_{i=0}^{\lfloor \tau k\rfloor-1} L^{\lfloor \tau k\rfloor-1-i}\Big|\p(\mathscr W_{0,(i+1+\beta)(\ell+t)})-L\p(\mathscr W_{0,(i+\beta)(\ell+t)})\Big|.
\end{align*}
Observing that $|\p(\mathscr W_{0,\beta(\ell+t)})-\p(\mathscr W_{0,\beta\ell})|\leq \lceil\beta t\rceil\p(A)$ and using Lemma~\ref{lem:inductive-step}, with $s=\lfloor\beta\ell\rfloor$ and $B=\X$, we obtain
$$
|\p(\mathscr W_{0,\beta(\ell+t)}-(1-\lfloor\beta\ell\rfloor\p(A))|\leq \lceil\beta t\rceil\p(A)+ \Xi_{A,\lfloor\beta\ell\rfloor}\leq \beta\Upsilon_A.
$$
Using the same argument that lead to \eqref{eq:important-relation} we may write
$$
\Big|\p(\mathscr W_{0,(i+1+\beta)(\ell+t)})-L\p(\mathscr W_{0,(i+\beta)(\ell+t)})\Big|\leq  \Upsilon_A\p\left(\bigcap_{j=0}^{i-1}\mathscr W_{j(\ell+t), \ell}\right)\leq \Upsilon_A (L+\Xi_{A,\ell})^i\leq \Upsilon_A (L+\Upsilon_A)^i.
$$
Hence, following the same argument used on the last computation of the proof of Lemma~\ref{lem:main-thing}, we obtain
\begin{align*}
\Big|\p(\mathscr W_{0,\tau k(\ell+t)})-L^{\lfloor \tau k\rfloor}(1-\lfloor\beta\ell\rfloor\p(A))\Big|&\leq \lfloor \tau k\rfloor\Upsilon_A (L+\Upsilon_A)^{\lfloor \tau k\rfloor-1}+\beta \Upsilon_A L^{\lfloor \tau k\rfloor}\\ &\leq \tau k\Upsilon_A(L+\Upsilon_A)^{\lfloor \tau k\rfloor-1}.
\end{align*}
Using the estimate above and \eqref{eq:approx-2} we have
\begin{align*}
\Big|\p(\mathscr W_{0,\tau n})-L^{\lfloor \tau k\rfloor}\Big|&\leq \lfloor\beta \ell\rfloor \p(A) L^{\lfloor \tau k\rfloor}+ \lceil\tau k\rceil\Upsilon_A(L+\Upsilon_A)^{\lfloor \tau k\rfloor-1}(1+L+\Upsilon_A)\\& \leq(3+\Upsilon_A)\lceil\tau k\rceil\Upsilon_A(L+\Upsilon_A)^{\lfloor \tau k\rfloor-1}.
\end{align*} %
The case $\lfloor \tau k\rfloor=0$ follows easily from the facts that $\left|\p(\mathscr W_{0,\tau n}))-\p(\mathscr W_{0,\tau k(\ell+t)})\right|\leq \lceil \tau b\rceil\p(A)=\p(A)$, $|\p(\mathscr W_{0,\tau k(\ell+t)})-\p(\mathscr W_{0,\tau k\ell})|\leq \lceil\tau k t\rceil\p(A) $ and  
$$
\Big|\p(\mathscr W_{0,\tau k\ell})-(1-\lfloor\tau k \ell\rfloor \p(A))\Big|\leq
\Xi_{A,\lfloor\tau k \ell\rfloor}
$$
which can be easily derived by using Lemma~\ref{lem:inductive-step}, with $s=\lfloor\tau k \ell\rfloor$ and $B=\X$.
\end{proof}

\section{Improved error estimates for EVLs}
\label{sec:EVL err}

\begin{proof}[Proof of Theorem~\ref{thm:sharp-error-EVL}]
First we recall some useful facts. By definition of $(u_n)_{n\in \N}$ and condition \ref{item:repeller}, we may write $n\p(U_n)\to\tau$ and $n\p(A_n)\to\theta\tau$. Recalling Remark~\ref{rem:sequences}, we observe that by hypothesis the sequences $(t_n)_{n\in\N}$ and $(k_n)_{n\in\N}$ are such that $t_n, k_n\to\infty$, as $n\to\infty$, and $t_nk_n=o(n)$. Hence, $t_n=o(n/k_n)$ and $k_n\ell_n\p(A_n)\to \theta\tau$, as $n\to\infty$.

Let $L_n=(1-\ell_n\p(A_n))$. Note that for $n$ large $0<L_n\leq 1$ and $\lim_{n\to\infty}L_n\to1$. We have:
\begin{align}
\label{eq:non-periodic-estimate}
\big|\p(M_n\leq u_n)-\e^{-\theta\tau}\big|&\leq \left|\p(M_n\leq u_n)-\p(\mathscr W_{0,n}(A_n))\right|+\left|\p(\mathscr W_{0,n}(A_n))-L_n^{k_n}\right|\nonumber\\&\quad+\left|L_n^{k_n}-\e^{-\theta\tau}\right|=:I_n+I\!\!I_n+I\!\!I\!\!I_n.
\end{align}
To estimate $I\!\!I\!\!I_n$, we start by noting that, by \eqref{eq:exponential-estimate2}, there exists $C$ such that
\begin{align}
\label{eq:computation-exponential-1-EVL}
\left|\e^{-k_n\ell_n\p(A_n)}-\e^{-\theta \tau}\right|&\leq \e^{-\theta\tau}\left[\left|\theta\tau-k_n\ell_n\p(A_n)\right|+\o(\left|\theta\tau-k_n\ell_n\p(A_n)\right|)\right]\nonumber\\
&\leq C  \e^{-\theta\tau}\left|\theta\tau-n\p(A_n)+k_n t_n\p(A_n)\right|=C   \e^{-\theta\tau}\alpha_n,
\end{align}
where $\alpha_n:=\left|\theta\tau-n\p(A_n)+k_n t_n\p(A_n)\right|$.

From  \eqref{eq:exponential-estimate}, we can easily obtain that $|(1-\frac x{k})^{k}-\e^{-x}|=\e^{-x}\;\frac{x^2}{2k}(1+\o(\frac1k))$, uniformly for $x$ in compact intervals. Since $\lim_{n\to\infty}k_n\ell_n\p(A_n)= \theta\tau$, then using these facts and \eqref{eq:computation-exponential-1-EVL}, there exist  $C', C''>0$ such that
\begin{align*}
\left|(1-\ell_n\p(A_n))^{k_n}-\e^{-k_n\ell_n\p(A_n)}\right|&=\e^{-k_n\ell_n\p(A_n)}\frac{(k_n\ell_n\p(A_n))^2}{2k_n}\left(1+\o\left(\frac{1}{k_n}\right)\right)\nonumber \\
&\leq C'\e^{-k_n\ell_n\p(A_n)}\frac {(\theta\tau)^2}{k_n} \\
&\leq C'' \e^{-\theta\tau} \left(\frac {(\theta\tau)^2}{k_n}+\frac {(\theta\tau)^2}{k_n}\alpha_n\right).
\end{align*}
Hence, there exists $C_3>0$, not depending on $n$, such that
\begin{equation}
\label{eq:second-term-EVL}
I\!\!I\!\!I_n=\left|L_n^{k_n}-\e^{-\theta\tau}\right|\leq C_3 \e^{-\theta\tau}\left(\alpha_n+\frac {(\theta\tau)^2}{k_n}\right).
\end{equation}

We next estimate $I\!\!I_n$.
Recall that, by definition, $R_n=R(A_n)\to\infty$, as $n\to\infty$. This implies that $\sum_{j=R_n}^{\infty}\gamma(j)\to0$, as $n\to\infty$.
Consequently we may write for $n$ sufficiently large:
\begin{align*}
\Upsilon_n:=\Upsilon_{A_n}&\leq t_n[\p(A_n)+M\gamma(t_n)]+M\ell_n \gamma (t_n)
+M\ell_n\left[\p(A_n)+M\gamma(t_n)\right]\sum_{j=R_n}^{\ell_n-1}\gamma(j)\\
&\quad +\ell_n^2\left[(\p(A_n))^2+M\p(A_n)\gamma(t_n)\right]\\
&\leq t_n\p(A_n)+3M^2\ell_n\gamma(t_n)+ \ell_n^2\p(A_n)^2+M\ell_n^2\p(A_n)\gamma(t_n)+M\ell_n\p(A_n)\sum_{j=R_n}^{\ell_n-1}\gamma(j).
\end{align*}

Hence, there exists $C>0$ such that for all $n$ sufficiently large we have
\begin{equation}
\label{eq:Upsilon-n-aperiodic}
k_n\Upsilon_{n}\leq C\left(k_n t_n\p(A_n)+n\gamma(t_n)\left(1+\frac{n\p(A_n)}{k_n}\right)+\frac{(n\p(A_n))^2}{k_n}+n\p(A_n)\sum_{j=R_n}^{\ell_n-1}\gamma(j)\right).
\end{equation}

By the properties of the sequences  $(k_n)_{n\in\N},(t_n)_{n\in\N}$, we have that $k_n\Upsilon_n\to0$, as $n\to\infty$. Moreover, since $L_n=1-\ell_n\p(A_n)\to1$, as $n\to\infty$ then it is clear that, for $n$ sufficiently large all assumptions of Corollary~\ref{cor:extra-assumption} are satisfied. Consequently, we can use Corollary~\ref{cor:extra-assumption}, \eqref{eq:second-term-EVL} and \eqref{eq:Upsilon-n-aperiodic} to obtain the existence of $C_2>0$ such that 
\begin{align*}
I\!\!I_n
&\leq C L_n^{k_n-1} \left(k_nt_n\p(A_n)+n\gamma(t_n)\right.
+\left.\frac{(n\p(A_n))^2}{k_n}+\frac{n\p(A_n)n\gamma(t_n)}{k_n}+n\p(A_n)\sum_{j=R_n}^{\ell_n-1}\gamma(j)\right)\\
&\leq C_2 \;\e^{-\theta\tau}\left(k_nt_n\p(A_n)+n\gamma(t_n)\left(1+\frac{\theta\tau}{k_n}\right)+\frac{(\theta\tau)^2}{k_n}+\theta\tau \sum_{j=R_n}^{\ell_n-1}\gamma(j)\right).
\end{align*}  
To estimate $I_n$, we start by noting that since $\lim_{n\to\infty}k_n\ell_n\p(A_n)=\theta\tau$, $\lim_{n\to\infty} L_n=1$ and $\lim_{n\to\infty}k_n\Upsilon_n=0$, then
\begin{equation}
\label{eq:L+Upsilon}
(L_n+\Upsilon_n)^{k_n-1}=\frac{(1-\ell_n\p(A_n)+\Upsilon_n)^{k_n}}{L_n+\Upsilon_n}=\frac{\left(1-\frac{k_n\ell_n\p(A_n)+k_n\Upsilon_n}{k_n}\right)^{k_n}}{L_n+\Upsilon_n}\xrightarrow[n\to\infty]{}\e^{-\theta\tau}.
\end{equation}
Let $n$ be sufficiently large in order to $\ell_n+t_n>q$ so that $n>j+(k_n-1)(\ell_n+t_n)$, for all $j=0, 1,\ldots, q$. Hence we can define $D_{n,j}=f^{-n+j+(k_n-1)(\ell_n+t_n)}(U_n\setminus A_n)$. Using first Proposition~\ref{prop:relation-balls-annuli-general}, then the statement of Lemma~\ref{lem:crucial-estimate} and finally \eqref{eq:L+Upsilon} together with the fact that  $\p(U_n\setminus A_n)\sim(1-\theta)\tau/n$ and stationarity, it follows that there exists $C_1>0$ such that
\begin{align*}
I_n=\big|\p(M_n\leq u_n)-\p(\mathscr W_{0,n}(A_n))\big|
&\leq \sum_{j=1}^{q} \p\left(\mathscr W_{0,n}(A_n)\cap f^{-n+j}(U_n\setminus A_n)\right)\\
&\leq \sum_{j=1}^{q} \p\left(\bigcap_{j=0}^{k_n-2}\mathscr W_{j(\ell_n+t_n), \ell_n}(A_n)\cap f^{-(k_n-1)(\ell_n+t_n)}(D_{n,j})\right)\\
&\leq q (L_n+\Upsilon_n)^{k_n-1}\p(D_{n,j})=q (L_n+\Upsilon_n)^{k_n-1}\p(U_n\setminus A_n)\\
&\leq C_1\;\e^{-\theta\tau}\frac\tau n.
\end{align*}
Putting all these estimates together into \eqref{eq:non-periodic-estimate} we get that there exists $C>0$ such that for all $n\in\N$,
\begin{align*}
\big|\p(M_n\leq u_n)-\e^{-\theta\tau}\big|\leq  C\e^{-\theta\tau}\Bigg(&\left|\theta\tau-n\p(A_n)\right| +k_nt_n\frac{\theta\tau} n+n\gamma(t_n)\left(1+\frac{\theta\tau}{k_n}\right)+\frac{(\theta\tau)^2}{k_n}\\&+\theta\tau\sum_{j=R_n}^{\ell_n-1}\gamma(j)\Bigg).
\end{align*}
\end{proof}

\section{Improved error estimates for HTS}
\label{sec:HTS err}

\begin{proof}[Proof of Theorem~\ref{thm:sharp-error-HTS}]
Let $n=\lfloor\p(B_\eps)^{-1}\rfloor$. Recall that $\p(A_\eps)\sim\theta\p(B_\eps)$ and, as discussed in Remark~\ref{rmk:decay consts}, we observe that by hypothesis $t_\eps$ and $k_\eps$ are such that $t_\eps, k_\eps\to\infty$, as $\eps\to0$, and $t_\eps k_\eps=o(n)$. 
Hence, for $\eps$ sufficiently small we have $\ell_\eps=\lfloor n/k_\eps\rfloor-t_\eps>0$. Moreover, on account of \ref{item:repeller}, $\lim_{\eps\to0}k_\eps\ell_\eps\p(A_\eps)=\theta$ and 
$L_\eps=1-\ell_\eps\p(A_\eps)\sim 1-\theta/k_\eps\to1$, as $\eps\to0$.  In particular, for $\eps>0$ sufficiently small, we have $0<L_\eps\leq1$.
 Now, observe that:
\begin{align}
\label{eq:triangular-HTS}
\Big|\p\Big(&r_{B_\eps(\zeta)}>\tfrac \tau{\p(B_\eps)}\Big)-\e^{-\theta \tau}\Big|\leq\big|\p(\mathscr W_{0,\tau n}(B_\eps)-\p(\mathscr W_{0,\tau n}(A_\eps))\big|+\Big|\p(\mathscr W_{0,\tau n}(A_\eps))-L^{\lfloor\tau k_\eps\rfloor}\Big|\nonumber\\&+\left|L_\eps^{\lfloor\tau k_\eps\rfloor}-\e^{-\theta\tau}\right|+\big|\p(\mathscr W_{0,\tau \p(B_\eps)^{-1}}(B_\eps)-\p(\mathscr W_{0,\tau n}(B_\eps))\big|=:I_\eps+I\!\!I_\eps+I\!\!I\!\!I_\eps+I\!V_\eps.
\end{align}

Now, we assume that $\tau\geq k_\eps^{-1}$, so that $\lfloor \tau k_\eps\rfloor>0$, and begin by estimating $I\!\!I\!\!I_\eps$. 

Since $\lim_{\eps\to0}\lfloor\tau k_\eps\rfloor\ell_\eps\p(A_\eps)=\theta\tau$, by \eqref{eq:exponential-estimate2} there exists $C$ independent of $\eps$ and $\tau$ such that
\begin{align}
\label{eq:computation-exponential-1}
\left|\e^{-\lfloor\tau k_\eps\rfloor\ell_\eps\p(A_\eps)}-\e^{-\theta \tau}\right|&\leq \e^{-\theta\tau}\left[\left|\theta\tau-\lfloor\tau k_\eps\rfloor\ell_\eps\p(A_\eps)\right|+\o(\left|\theta\tau-\lfloor\tau k_\eps\rfloor\ell_\eps\p(A_\eps)\right|)\right]\nonumber\\
&\leq C \tau  \e^{-\theta\tau}\left|\theta-\frac{\p(A_\eps)}{\p(B_\eps)}+k_\eps t_\eps\p(A_\eps)\right|=C \tau  \e^{-\theta\tau}\alpha_\eps.
\end{align}
As before, from  \eqref{eq:exponential-estimate}, we can easily obtain $|(1-\frac x{k})^{k}-\e^{-x}|=\e^{-x}\;\frac{x^2}{2k}(1+\o(\frac1k))$, uniformly for $x$ in compact intervals. Since $\lim_{\eps\to0}\lfloor\tau k_\eps\rfloor\ell_\eps\p(A_\eps)=\theta\tau$, then using these facts and \eqref{eq:computation-exponential-1}, we have that there exist  $C', C''>0$ independent of $\eps$ and $\tau$ such that
\begin{align*}
\left|(1-\ell_\eps\p(A_\eps))^{\lfloor\tau k_\eps\rfloor}-\e^{-\lfloor\tau k_\eps\rfloor\ell_\eps\p(A_\eps)}\right|&=\e^{-\lfloor\tau k_\eps\rfloor\ell_\eps\p(A_\eps)}\frac{(\lfloor\tau k_\eps\rfloor\ell_\eps\p(A_\eps))^2}{2\lfloor\tau k_\eps\rfloor}\left(1+\o\left(\frac{1}{\lfloor\tau k_\eps\rfloor}\right)\right)\nonumber \\
&\leq C'\e^{-\lfloor\tau k_\eps\rfloor\ell_\eps\p(A_\eps)}\frac {\tau}{k_\eps} \\
&\leq C'' \e^{-\theta\tau} \left(\frac {\tau}{k_\eps}+\frac {\tau^2}{k_\eps}\alpha_\eps\right).
\end{align*}
Hence, there exists $C_3>0$, independent of  $\eps$, but not of $\tau$, such that
\begin{equation}
\label{eq:second-term}
I\!\!I\!\!I_\eps=\left|L_\eps^{\lfloor\tau k_\eps\rfloor}-\e^{-\theta\tau}\right|\leq C_3\; \e^{-\theta\tau}\left(\tau\alpha_\eps+\frac {\tau}{k_\eps}+\frac {\tau^2}{k_\eps}\alpha_\eps\right).
\end{equation}

To estimate $I\!\!I_\eps$, we need to use Corollary~\ref{cor:tau-included}. Note that the hypothesis on $t_\eps$ and $k_\eps$ recalled in the beginning of the proof guarantee that the assumptions of Corollary~\ref{cor:tau-included} are satisfied.
Also recall that, by definition, we have that $R_\eps\to \infty$, as $\eps\to0$. This implies that $\sum_{j=R_\eps}^{\infty}\gamma(j)\to0$, as $\eps\to0$. Consequently we may write for $\eps>0$ sufficiently small:
\begin{align*}
\Upsilon_{A_\eps}&\leq t_\eps[\p(A_\eps)+M\gamma(t_\eps)]+M\ell_\eps \gamma (t_\eps)
+M\ell_\eps\left[\p(A_\eps)+M\gamma(t_\eps)\right]\sum_{j=R_\eps}^{\ell_\eps-1}\gamma(j)\\
&\quad +\ell_\eps^2\left[(\p(A_\eps))^2+M\p(A_\eps)\gamma(t_\eps)\right]\\
&\leq t_\eps\p(A_\eps)+3M^2\ell_\eps\gamma(t_\eps)+M\ell_\eps\gamma(t_\eps)\ell_\eps\p(A_\eps)+ \ell_\eps^2\p(A_\eps)^2+M\ell_\eps\p(A_\eps)\sum_{j=R_\eps}^{\ell_\eps-1}\gamma(j).
\end{align*}

Hence, there exists $C>0$ such that for all $\eps>0$  we have
\begin{equation}
\label{eq:Upsilon-estimate}
k_\eps\Upsilon_{A_\eps}\leq C\left(k_\eps t_\eps\p(A_\eps)+(\p(B_\eps))^{-1}\gamma(t_\eps)+ \frac1{k_\eps}+\sum_{j=R_\eps}^{\ell_\eps-1}\gamma(j)\right)=:C\Gamma_\eps.
\end{equation}
By the choices of $k_\eps, t_\eps$ it is clear that $k_\eps\Upsilon_{A_\eps}\to0$, as $\eps\to0$. 
Recalling  that $L_\eps\sim 1-\theta/k_\eps\to1$, as $\eps\to0$ and using \eqref{eq:second-term}, there exists $C>0$ such that
\begin{align}
\label{eq:L+Upsilon-2}
(L_\eps+\Upsilon_{A_\eps})^{\lfloor\tau k_\eps\rfloor-1}&=\frac{L_\eps^{\lfloor\tau k_\eps\rfloor}}{L_\eps+\Upsilon_{A_\eps}}\left(1+\frac{\lfloor\tau k_\eps\rfloor\Upsilon_{A_\eps}L_\eps^{-1}}{\lfloor\tau k_\eps\rfloor}\right)^{\lfloor\tau k_\eps\rfloor}\leq  \frac{L_\eps^{\lfloor\tau k_\eps\rfloor}}{L_\eps+\Upsilon_{A_\eps}}\e^{\lfloor\tau k_\eps\rfloor\Upsilon_{A_\eps}L_\eps^{-1}}\nonumber \\&\leq C \left(\tau\alpha_\eps+\frac {\tau}{k_\eps}+\frac {\tau^2}{k_\eps}\alpha_\eps\right) \e^{-(\theta- k_\eps\Upsilon_{A_\eps}L_\eps^{-1})\tau}.
\end{align}
Applying Corollary~\ref{cor:tau-included}, we have 
\begin{align*}
\Big|\p(\mathscr W_{0,\tau n}(A_\eps))-L_\eps^{\lfloor\tau k_\eps\rfloor}\Big|&\leq(3+\Upsilon_{A_\eps})\lceil\tau k_\eps\rceil\Upsilon_{A_\eps}(L_\eps+\Upsilon_{A_\eps})^{\lfloor \tau k\rfloor-1}.
\end{align*}
Using equation~\eqref{eq:L+Upsilon-2}, there exists $C_2>0$ independent of $\eps$ and $\tau$ such that
\begin{equation}
\label{eq:ring-estimate}
I\!\!I_\eps=\Big|\p(\mathscr W_{0,\tau n}(A_\eps))-L^{\lfloor\tau k_\eps\rfloor}\Big|\leq C_2 \left(\tau^2\alpha_\eps\Gamma_\eps+\frac {\tau^2}{k_\eps}\Gamma_\eps+\frac {\tau^3}{k_\eps}\alpha_\eps\Gamma_\eps\right) \e^{ -(\theta- k_\eps\Upsilon_{A_\eps}L_\eps^{-1})\tau}.
\end{equation}
Let $\eps$ be sufficiently small in order to $\ell_\eps+t_\eps>q$ so that $\lfloor\tau n\rfloor>j+(\lfloor\tau k_\eps\rfloor-1)(\ell_\eps+t_\eps)$, for all $j=0, 1,\ldots, q$. Hence we can         define $D_{\eps,j}=f^{-\lfloor\tau n\rfloor+j+(\lfloor\tau k_\eps\rfloor-1)(\ell_\eps+t_\eps)}(U_\eps\setminus A_\eps)$. Using first Proposition~\ref{prop:relation-balls-annuli-general}, then the statement of Lemma~\ref{lem:crucial-estimate} and finally \eqref{eq:L+Upsilon-2} together with the fact that  $\p(U_\eps\setminus A_\eps)\sim(1-\theta)\tau/n$ and stationarity, it follows that there exists $C_1>0$ such that
\begin{align*}
I_\eps=\big|\p(\mathscr W_{0,\tau n}(B_\eps)&-\p(\mathscr W_{0,\tau n}(A_\eps))\big|= \big|\p(\mathscr W_{0,\lfloor\tau n\rfloor}(B_\eps)-\p(\mathscr W_{0,\lfloor\tau n\rfloor}(A_\eps))\big|\\
&\leq \sum_{j=1}^{q} \p\left(\mathscr W_{0,\lfloor\tau n\rfloor}(A_\eps)\cap f^{-n+j}(U_\eps\setminus A_\eps)\right)\\
&\leq \sum_{j=1}^{q} \p\left(\bigcap_{j=0}^{\lfloor\tau k_\eps\rfloor-2}\mathscr W_{j(\ell_\eps+t_\eps), \ell_\eps}(A_\eps)\cap f^{-(\lfloor\tau k_\eps\rfloor-1)(\ell_\eps+t_\eps)}(D_{\eps,j})\right)\\
&\leq q (L_\eps+\Upsilon_\eps)^{\lfloor\tau k_\eps\rfloor-1}\p(D_{\eps,j})=q (L_\eps+\Upsilon_\eps)^{\lfloor\tau k_\eps\rfloor-1}\p(U_\eps\setminus A_\eps)\\
&\leq C_1\;\p(B_\eps) \left(\tau\alpha_\eps+\frac {\tau}{k_\eps}+\frac {\tau^2}{k_\eps}\alpha_\eps\right) \e^{ -(\theta- k_\eps\Upsilon_{A_\eps}L_\eps^{-1})\tau}.
\end{align*}

Recalling the facts that $\lfloor\tau n\rfloor\leq \lfloor\tau\p(B_\eps)^{-1}\rfloor\leq \lfloor\tau n\rfloor +\tau$, $t_\eps=\o(\ell_\eps)$, $\gamma(t_\eps)=\o(\p(B_\eps))$, using Lemma~\ref{lem:time-gap}, with $s=0$ and $t'= \lfloor\tau \p(B_\eps)^{-1}\rfloor-\lfloor\tau n\rfloor\leq \tau$ and Lemma~\ref{lem:crucial-estimate}, with $D=\X$, we obtain that for $\eps$ sufficiently small:
\begin{align*}
\big|\p(&\mathscr W_{0,\tau \p(B_\eps)^{-1}}(B_\eps)-\p(\mathscr W_{0,\tau n}(B_\eps))\big|\leq \tau (\p(B_\eps)+M\gamma(t_\eps))\p(\mathscr W_{0,\tau n-t_\eps}(B_\eps))\\
&\leq \tau (\p(B_\eps)+M\gamma(t_\eps))\p\left(\bigcap_{j=0}^{\lfloor\tau k_\eps\rfloor-1}\mathscr W_{j(\ell_\eps+t_\eps), \ell_\eps}(A_\eps)\right)\leq 2M \tau \p(B_\eps) (L_\eps+\Upsilon_{A_\eps})^{\lfloor\tau k_\eps\rfloor}.
\end{align*}
Again, using \eqref{eq:L+Upsilon-2}, there exists $C_4>0$ independent of $\eps$ and $\tau$ such that
\begin{equation*}
I\!V_\eps=\big|\p(\mathscr W_{0,\tau \p(B_\eps)^{-1}}(B_\eps)-\p(\mathscr W_{0,\tau n}(B_\eps))\big|\leq C_4 \p(B_\eps) \left(\tau^2\alpha_\eps+\frac {\tau^2}{k_\eps}+\frac {\tau^3}{k_\eps}\alpha_\eps\right) \e^{ -(\theta- k_\eps\Upsilon_{A_\eps}L_\eps^{-1})\tau}.
\end{equation*}
Assuming that $\tau\geq 1$ and combining the estimates on $I_\eps, I\!\!I_\eps, I\!\!I\!\!I_\eps, I\!V_\eps$,  it follows that by taking $C_5=\max\{C_1,\ldots, C_4\}$, which is independent of $\eps$ and  $\tau$, we have:
$$
\Big|\p\left(r_{B_\eps(\zeta)}>\frac \tau{\p(B_\eps)}\right)-\e^{-\theta \tau}\Big|\leq 4C_5 \left(\tau^2\alpha_\eps\Gamma_\eps+\frac {\tau^2}{k_\eps}\Gamma_\eps+\frac {\tau^3}{k_\eps}\alpha_\eps\Gamma_\eps\right) \e^{ -(\theta- k_\eps\Upsilon_{A_\eps}L_\eps^{-1})\tau}.
$$
Note that if $k_\eps^{-1}\leq \tau <1$ then \eqref{eq:ring-estimate} can be rewritten in the following away:
\begin{equation*}
\Big|\p(\mathscr W_{0,\tau n}(A_\eps))-L^{\lfloor\tau k_\eps\rfloor}\Big|\leq C \Gamma_\eps \e^{ -\theta \tau},
\end{equation*}
for some $C>0$ independent of $\eps$. Hence, in this case we may write that there exists some constant $C>0$ independent of $\eps$ such that
$\Big|\p\left(r_{B_\eps(\zeta)}>\frac \tau{\p(B_\eps)}\right)-\e^{-\theta \tau}\Big|\leq C (\alpha_\eps+\Gamma_\eps) \e^{-\theta \tau}$.
Now, we consider the case when $\tau<k_\eps^{-1}$. By Corollary~\ref{cor:tau-included}, in this case $\left|\p(\mathscr W_{0,\tau n}(A_\eps))- (1-\lfloor\tau k_\eps\ell_\eps\rfloor\p(A_\eps))\right|\leq \tau k_\eps \Upsilon_\eps$. Also note that there exist $0<\xi<\theta\tau$ and $C>0$ depending on $\eps$ such that 
\begin{align*}
\left|\e^{-\theta\tau}-(1-\lfloor\tau k_\eps\ell_\eps\rfloor\p(A_\eps))\right|&=\left|1-\theta\tau+\frac{\e^{-\xi}}{2}(\theta\tau)^2-(1-\lfloor\tau k_\eps\ell_\eps\rfloor\p(A_\eps))\right|\\
&\leq C\alpha_\eps+\frac{\e^{-\xi}}{2}(\theta\tau)^2.
\end{align*}
Since $\tau<k_\eps^{-1}<\Gamma_\eps$, for $\eps>0$ sufficiently small we have $\e^{\theta\tau}\leq 1$, so we can write, again, that there exists $C>0$ independent of  $\eps$ such that
$$
\Big|\p\left(r_{B_\eps(\zeta)}>\frac \tau{\p(B_\eps)}\right)-\e^{-\theta \tau}\Big|\leq C (\alpha_\eps+\Gamma_\eps) \e^{-\theta \tau}.
$$
\end{proof}

 \section{An application: Rychlik systems}
\label{sec:Rych}

To apply the most powerful theorems in this paper, i.e., Theorems~\ref{thm:sharp-error-EVL} and \ref{thm:sharp-error-HTS}, we would like a dynamical system $f:\X\to \X$ with a measure $\mu$ which satisfies \ref{item:U-ball} and \ref{item:repeller} and moreover has decay of correlations, at more than quadratic rate, against $L^1$ observables.  In this section we give a natural class of examples of such a system, particular examples to keep in mind here are piecewise smooth, uniformly expanding, full-branched interval maps with an absolutely continuous invariant measure.  The main idea is to use maps to which the seminal paper \cite{R83} applies, along with some extra information on the smoothness of potentials which will give us continuity of measure.  The maps we define below are certainly not in the most general form possible, but we will make some restrictions for expository reasons.  Note that the setup of \cite{K12} also applies to this class of examples.  We also recall that our theory applies to higher dimensional examples, such as those studied in \cite{S00}.

\subsection{Interval maps modelled by a full shift}

We will consider our maps $F:\cup_iC^i\to X$ where $X$ is an interval, $\cup_iC^i\subset X$ is an at most countable union of open intervals and $F: C^i\to X$ is a bijection.  Let $X^\infty=\{x\in X:F^n(x) \text{ is defined for all } n\in \N\}$.  Given a sequence of natural numbers  $i_0, \ldots, i_{n-1}$, the collection of points $x$ which have $F^k(x)\in C^{i_k}$ for $k=0, \ldots, n-1$ is the corresponding \emph{$n$-cylinder}.  Let $\P_n$ denote the collection of all such cylinders.  We will further assume that for any $x, y\in X^\infty$, if $x\neq y$ then there exists $n\in \N$ such that $x$ and $y$ are in different $n$-cylinders.  Since such maps are nicely modelled by the full shift, we denote the class of these maps by $\FS$.  

\subsection{Thermodynamic formalism for $\FS$ maps}
Given a  potential $\Phi:X\to \R$, we define the \emph{$n$-th variation} as 
$$V_n(\Phi):=\sup_{C_n\in \P_n}\sup_{x, y\in C_n}\{|\Phi(x)-\Phi(y)|\}.$$
The potential is said to have \emph{summable variations} if $\sum_{n\ge 1} V_n(\Phi)<\infty$, and 
be \emph{locally H\"older} if there exists $\alpha>0$ such that $V_n(\Phi)=O(e^{-\alpha n})$. 

Define $$Z_n(\Phi):=\sum_{F^n x=x}e^{S_n\Phi(x)}$$
where $S_n\Phi$ is the ergodic sum $\Phi+\Phi\circ F+\cdots+ \Phi\circ F^{n-1}$.
Assuming that $\Phi$ has summable variations then we define the \emph{pressure} as
$$P(\Phi):=\lim_{n\to \infty}\frac1n\log Z_n(\Phi).$$

 \begin{theorem}
 Suppose that $\Phi$ is a locally H\"older  potential and $P(\Phi)<\infty$.  Then there exists a $(\Phi -P(\Phi))$-conformal probability measure $m_\Phi$ and a density $\rho_\Phi$ such that $d\mu_\Phi=\rho_\Phi dm_\Phi$ defines an invariant probability measure.  Moreover,  $\log\rho_\Phi$ is locally H\"older; in particular $\rho_\Phi$ is bounded away from zero and infinity on cylinders. If  $-\int\Phi~d\mu<\infty$ then $\mu_\Phi$ is an equilibrium state for $\Phi$.
 \label{thm:RPF}
\end{theorem} 

This theorem is part of \cite[Theorem 1]{S01} with the smoothness of $\rho_\Phi$ discussed in Remark 2 of that paper.  Specifically for the full shift case we consider here, see also \cite{MU01} and \cite{S03}.
Note that as proved in Sarig's thesis, this theorem also holds under summable variations.

\subsection{Continuity of measures: conditions \ref{item:U-ball} and \ref{item:repeller}}
\label{ssec:cty}
Theorem~\ref{thm:RPF} shows that $\rho_\Phi$ is fairly smooth in a symbolic sense (i.e., in terms of the metric on $X^\infty$ induced from the cylinder structure).  However, this type of property also passes to the usual metric on the interval: namely, for $x\in X^\infty$
\begin{equation}
\lim_{\delta\to 0} \frac{\mu_\Phi(B_\delta(x))}{m_\Phi(B_\delta(x))} = \rho_\Phi(x).
\label{eq:Radon N}
\end{equation}
This follows since by the H\"older continuity of $\log\rho_\Phi$, for any $\eps>0$ there exists $n$ such that for $x,y$ in the same $n$-cylinder, $e^{-\eps}\rho_\Phi(x)\le \rho_\Phi(y)\le e^\eps\rho_\Phi(x)$.  Hence, once $\delta$ is so small that $B_\delta(x)$ is contained in the $n$-cylinder at $x$, 
$$e^{-\eps}\rho_\Phi(x)\le \frac{\mu_\Phi(B_\delta(x))}{m_\Phi(B_\delta(x))}\le e^{\eps}\rho_\Phi(x).$$

A further remark is that since $m_\Phi$ is non-atomic and gives any open set positive measure, for $z\in X^\infty$, $\delta\mapsto m_\Phi(B_\delta(z))$ is continuous at 0.  Moreover, by \eqref{eq:Radon N}, $\delta\mapsto \mu_\Phi(B_\delta(z))$ is continuous at 0. So clearly \ref{item:U-ball} holds.  For \ref{item:repeller}, we can use the conformality of $m_\Phi$ to prove that in fact $\theta=1-e^{S_p\Phi(z)-pP(\Phi)}$.

\subsection{Rychlik conditions}
In order to get the nice decay of correlations for maps in $\FS$ with good potentials, we use \cite{R83}.  This requires the further assumption that  $Var(e^\Phi)<\infty$, where here we mean $Var$ in the classical sense of variations.  It is easy to show that this implies that  $P(\Phi)<\infty$.  In fact, as we see in the theorem below, we require more: if we define $\overline\Phi:\cup_i\overline{C^i}\to [-\infty, \infty)$ by
\begin{equation*}
\overline\Phi(x) =
\begin{cases} \Phi(x) & \text{ if } x\in \cup_iC^i,\\
-\infty & \text{ if } x\in \cup_i\partial C^i.
\end{cases}
\end{equation*}

\begin{theorem}[\cite{R83}]
Suppose that $\Phi$ is a locally H\"older potential with $Var(e^{\overline\Phi})<\infty$ and $\Phi<0$.  Then the measure $\mu_\Phi$ obtained in Theorem~\ref{thm:RPF} has exponential decay of correlations for BV against $L^1$ observables.
\end{theorem}

We say that a system $(X, F, \Phi)$ satisfying the conditions in the theorem is a \emph{Rychlik system}.

Notice that a characteristic function on an annulus has BV-norm bounded by 4, so in Theorem~\ref{thm:sharp-error-EVL} (Theorem~\ref{thm:sharp-error-HTS}) we can take $M$ to be 4 for $n$ large enough (for $\eps$ small enough).  

A simple example of an application of this theorem is to the case that there exist $\lambda>1$ and $K\ge 1$ such that $F$ is $C^{1+Lip}$ and for each $i$, $F:C_i\to X$ has bounded distortion: $|DF(x)|/|DF(y)|\le K$ for $x, y\in C_i$.  Then set $\Phi(x)=-\log|DF(x)|$.  Our conditions guarantee local H\"older continuity for $\Phi$ and $P(\Phi)<\infty$ as well as $Var(e^{\overline\Phi})\le (C+K)m(X)<\infty$ (where $C$ is the Lipschitz constant) and $\Phi<0$.  If moreover, the union of the domains $C_i$ is equal to $X$ up to sets of Lebesgue measure zero, then $m_\Phi$ is Lebesgue and the resulting measure $\mu_\Phi$ is an equilibrium state which is also probability measure absolutely continuous w.r.t. Lebesgue (an \emph{acip}).

\subsection{Discussion of $R_\eps$}
Given a Rychlik system $(\X, f, \mu)$ as above,
we know from \cite{STV02} that if $\mu$ is an ergodic $F$-invariant measure with positive entropy, then for a typical point $\zeta$, $R_\eps=R(A_\eps(\zeta))$ grows like $\eps^{-d}$ for some $d\in (0,1]$ (in fact $d$ is the dimension of the measure $\mu$).  Also it is easy to see that for $\zeta$ a repelling periodic point, $R(A_\eps(\zeta))$ is at least of order $\eps^{-d'}$ where $d'$ depends on the strength of repulsion at $\zeta$.  An argument showing that indeed $R(A_\eps(\zeta))\to \infty$ for all points $\zeta$ was given in \cite[Section 6]{FFT12}.

\bibliographystyle{amsalpha}

\bibliography{ErrorTerms.bib}

\providecommand{\bysame}{\leavevmode\hbox to3em{\hrulefill}\thinspace}
\providecommand{\MR}{\relax\ifhmode\unskip\space\fi MR }
\providecommand{\MRhref}[2]{%
  \href{http://www.ams.org/mathscinet-getitem?mr=#1}{#2}
}
\providecommand{\href}[2]{#2}
\begin{thebibliography}{AFLV11}

\bibitem[Aba01]{A01}
Miguel Abadi, \emph{Exponential approximation for hitting times in mixing
  processes}, Math. Phys. Electron. J. \textbf{7} (2001), Paper 2, 19 pp.
  (electronic). \MR{1871384 (2002h:60069)}

\bibitem[Aba04]{A04}
\bysame, \emph{Sharp error terms and necessary conditions for exponential
  hitting times in mixing processes}, Ann. Probab. \textbf{32} (2004), no.~1A,
  243--264. \MR{MR2040782 (2004m:60042)}

\bibitem[ACS00]{ACS00}
Valentin Afraimovich, Jean-Rene Chazottes, and Beno{\^{\i}}t Saussol,
  \emph{Local dimensions for {P}oincar\'e recurrences}, Electron. Res. Announc.
  Amer. Math. Soc. \textbf{6} (2000), 64--74 (electronic). \MR{1777857
  (2001d:37021)}

\bibitem[AFLV11]{AFLV11}
Jos{\'e}~F. Alves, Jorge~M. Freitas, Stefano Luzzatto, and Sandro Vaienti,
  \emph{From rates of mixing to recurrence times via large deviations}, Adv.
  Math. \textbf{228} (2011), no.~2, 1203--1236. \MR{2822221}

\bibitem[AFV12]{AFV12}
Hale Ayta{\c c}, Jorge~Milhazes Freitas, and Sandro Vaienti, \emph{Laws of rare
  events for deterministic and random dynamical systems}, To appear in
  Transactions of the American Mathematical Society, (Preprint
  arXiv:1207.5188), 2012.

\bibitem[AG01]{AG01}
M.~Abadi and A.~Galves, \emph{Inequalities for the occurrence times of rare
  events in mixing processes. {T}he state of the art}, Markov Process. Related
  Fields \textbf{7} (2001), no.~1, 97--112, Inhomogeneous random systems
  (Cergy-Pontoise, 2000). \MR{MR1835750}

\bibitem[AL13]{AL13}
Miguel Abadi and Rodrigo Lambert, \emph{The distribution of the short-return
  function}, Nonlinearity \textbf{26} (2013), no.~5, 1143--1162. \MR{3043376}

\bibitem[Ald82]{A82}
David~J. Aldous, \emph{Markov chains with almost exponential hitting times},
  Stochastic Process. Appl. \textbf{13} (1982), no.~3, 305--310. \MR{671039
  (84b:60096a)}

\bibitem[AS11]{AS11}
Miguel Abadi and Benoit Saussol, \emph{Hitting and returning to rare events for
  all alpha-mixing processes}, Stochastic Process. Appl. \textbf{121} (2011),
  no.~2, 314--323. \MR{2746177}

\bibitem[AV08]{AV08}
Miguel Abadi and Sandro Vaienti, \emph{Large deviations for short recurrence},
  Discrete Contin. Dyn. Syst. \textbf{21} (2008), no.~3, 729--747.
  \MR{MR2399435 (2009j:37015)}

\bibitem[AV09]{AV09}
Miguel Abadi and Nicolas Vergne, \emph{Sharp error terms for return time
  statistics under mixing conditions}, J. Theoret. Probab. \textbf{22} (2009),
  no.~1, 18--37. \MR{MR2472003 (2010f:60068)}

\bibitem[CHM91]{CHM91}
Michael~R. Chernick, Tailen Hsing, and William~P. McCormick, \emph{Calculating
  the extremal index for a class of stationary sequences}, Adv. in Appl.
  Probab. \textbf{23} (1991), no.~4, 835--850. \MR{MR1133731 (93c:60073)}

\bibitem[Col01]{C01}
P.~Collet, \emph{Statistics of closest return for some non-uniformly hyperbolic
  systems}, Ergodic Theory Dynam. Systems \textbf{21} (2001), no.~2, 401--420.
  \MR{MR1827111 (2002a:37038)}

\bibitem[Dol98]{D98}
Dmitry Dolgopyat, \emph{On decay of correlations in {A}nosov flows}, Ann. of
  Math. (2) \textbf{147} (1998), no.~2, 357--390. \MR{1626749 (99g:58073)}

\bibitem[DY06]{DY06}
Mark~F. Demers and Lai-Sang Young, \emph{Escape rates and conditionally
  invariant measures}, Nonlinearity \textbf{19} (2006), no.~2, 377--397.
  \MR{2199394 (2006i:37051)}

\bibitem[Fel50]{F50}
William Feller, \emph{An {I}ntroduction to {P}robability {T}heory and {I}ts
  {A}pplications. {V}ol. {I}}, John Wiley \& Sons Inc., New York, N.Y., 1950.
  \MR{MR0038583 (12,424a)}

\bibitem[FF08]{FF08}
Ana Cristina~Moreira Freitas and Jorge~Milhazes Freitas, \emph{On the link
  between dependence and independence in extreme value theory for dynamical
  systems}, Statist. Probab. Lett. \textbf{78} (2008), no.~9, 1088--1093.
  \MR{MR2422964 (2009e:37006)}

\bibitem[FFT10]{FFT10}
Ana Cristina~Moreira Freitas, Jorge~Milhazes Freitas, and Mike Todd,
  \emph{Hitting time statistics and extreme value theory}, Probab. Theory
  Related Fields \textbf{147} (2010), no.~3-4, 675--710. \MR{2639719
  (2011g:37015)}

\bibitem[FFT11]{FFT11}
\bysame, \emph{Extreme value laws in dynamical systems for non-smooth
  observations}, J. Stat. Phys. \textbf{142} (2011), no.~1, 108--126.
  \MR{2749711 (2012a:60149)}

\bibitem[FFT12]{FFT12}
\bysame, \emph{The extremal index, hitting time statistics and periodicity},
  Adv. Math. \textbf{231} (2012), no.~5, 2626--2665. \MR{2970462}

\bibitem[FFT13]{FFT13}
\bysame, \emph{The compound {P}oisson limit ruling periodic extreme behaviour
  of non-uniformly hyperbolic dynamics}, Comm. Math. Phys. \textbf{321} (2013),
  no.~2, 483--527. \MR{3063917}

\bibitem[FP12]{FP12}
Andrew Ferguson and Mark Pollicott, \emph{Escape rates for gibbs measures},
  Ergod. Theory Dynam. Systems \textbf{32} (2012), 961--988.

\bibitem[Fre13]{F13}
Jorge~Milhazes Freitas, \emph{Extremal behaviour of chaotic dynamics}, Dyn.
  Syst. \textbf{28} (2013), no.~3, 302--332.

\bibitem[GS97]{GS97}
A.~Galves and B.~Schmitt, \emph{Inequalities for hitting times in mixing
  dynamical systems}, Random Comput. Dynam. \textbf{5} (1997), no.~4, 337--347.
  \MR{1483874 (98i:60017)}

\bibitem[HN14]{HN14}
Mark Holland and Matthew Nicol, \emph{Speed of convergence to an extreme value
  distribution for non-uniformly hyperbolic dynamical systems}, Preprint, 2014.

\bibitem[HNT12]{HNT12}
Mark Holland, Matthew Nicol, and Andrei T{\"o}r{\"o}k, \emph{Extreme value
  theory for non-uniformly expanding dynamical systems}, Trans. Amer. Math.
  Soc. \textbf{364} (2012), no.~2, 661--688. \MR{2846347 (2012k:37064)}

\bibitem[HSV99]{HSV99}
Masaki Hirata, Beno{\^{\i}}t Saussol, and Sandro Vaienti, \emph{Statistics of
  return times: a general framework and new applications}, Comm. Math. Phys.
  \textbf{206} (1999), no.~1, 33--55. \MR{MR1736991 (2001c:37007)}

\bibitem[HV10]{HV10}
Nicolai Haydn and Sandro Vaienti, \emph{The {R}\'enyi entropy function and the
  large deviation of short return times}, Ergodic Theory Dynam. Systems
  \textbf{30} (2010), no.~1, 159--179. \MR{2586350 (2011a:37017)}

\bibitem[HW79]{HW79}
W.~J. Hall and Jon~A. Wellner, \emph{The rate of convergence in law of the
  maximum of an exponential sample}, Statist. Neerlandica \textbf{33} (1979),
  no.~3, 151--154. \MR{552259 (80k:60029)}

\bibitem[Kel12]{K12}
Gerhard Keller, \emph{Rare events, exponential hitting times and extremal
  indices via spectral perturbation}, Dynamical Systems \textbf{27} (2012),
  no.~1, 11--27.

\bibitem[KL09]{KL09}
Gerhard Keller and Carlangelo Liverani, \emph{Rare events, escape rates and
  quasistationarity: some exact formulae}, J. Stat. Phys. \textbf{135} (2009),
  no.~3, 519--534. \MR{2535206 (2011a:37012)}

\bibitem[Lea74]{L73}
M.~R. Leadbetter, \emph{On extreme values in stationary sequences}, Z.
  Wahrscheinlichkeitstheorie und Verw. Gebiete \textbf{28} (1973/74), 289--303.
  \MR{MR0362465 (50 \#14906)}

\bibitem[LLR83]{LLR83}
M.~R. Leadbetter, Georg Lindgren, and Holger Rootz{{\'e}}n, \emph{Extremes and
  related properties of random sequences and processes}, Springer Series in
  Statistics, Springer-Verlag, New York, 1983. \MR{MR691492 (84h:60050)}

\bibitem[LN89]{LN89}
M.~R. Leadbetter and S.~Nandagopalan, \emph{On exceedance point processes for
  stationary sequences under mild oscillation restrictions}, Extreme value
  theory ({O}berwolfach, 1987), Lecture Notes in Statist., vol.~51, Springer,
  New York, 1989, pp.~69--80. \MR{MR992049 (90k:60065)}

\bibitem[MS01]{MS01}
William~P. McCormick and Lynne Seymour, \emph{Rates of convergence and
  approximations to the distribution of the maximum of chain-dependent
  sequences}, Extremes \textbf{4} (2001), no.~1, 23--52. \MR{1876178
  (2002i:60106)}

\bibitem[MU01]{MU01}
R.~Daniel Mauldin and Mariusz Urba{\'n}ski, \emph{Gibbs states on the symbolic
  space over an infinite alphabet}, Israel J. Math. \textbf{125} (2001),
  93--130. \MR{1853808 (2002k:37048)}

\bibitem[Res08]{R08}
Sidney~I. Resnick, \emph{Extreme values, regular variation and point
  processes}, Springer Series in Operations Research and Financial Engineering,
  Springer, New York, 2008, Reprint of the 1987 original. \MR{2364939
  (2008h:60002)}

\bibitem[Ryc83]{R83}
Marek Rychlik, \emph{Bounded variation and invariant measures}, Studia Math.
  \textbf{76} (1983), no.~1, 69--80. \MR{MR728198 (85h:28019)}

\bibitem[Sar01]{S01}
Omri~M. Sarig, \emph{Thermodynamic formalism for null recurrent potentials},
  Israel J. Math. \textbf{121} (2001), 285--311. \MR{1818392 (2001m:37059)}

\bibitem[Sar03]{S03}
Omri Sarig, \emph{Existence of {G}ibbs measures for countable {M}arkov shifts},
  Proc. Amer. Math. Soc. \textbf{131} (2003), no.~6, 1751--1758 (electronic).
  \MR{1955261 (2004b:37056)}

\bibitem[Sau00]{S00}
Beno{\^{\i}}t Saussol, \emph{Absolutely continuous invariant measures for
  multidimensional expanding maps}, Israel J. Math. \textbf{116} (2000),
  223--248. \MR{1759406 (2001e:37037)}

\bibitem[Smi82]{S82}
Richard~L. Smith, \emph{Uniform rates of convergence in extreme-value theory},
  Adv. in Appl. Probab. \textbf{14} (1982), no.~3, 600--622. \MR{665296
  (84h:60054)}

\bibitem[STV02]{STV02}
B.~Saussol, S.~Troubetzkoy, and S.~Vaienti, \emph{Recurrence, dimensions, and
  {L}yapunov exponents}, J. Statist. Phys. \textbf{106} (2002), no.~3-4,
  623--634. \MR{MR1884547 (2003a:37007)}

\end{thebibliography}

\end{document}